\documentclass[bj,numbers]{imsart}
\RequirePackage{amsthm,amsmath,amsfonts,amssymb}
\usepackage[capitalize]{cleveref}
\usepackage{mathrsfs}

\newcommand{\mbR}{\mathbb{R}}

\newcommand{\cB}{\mathcal{B}}
\newcommand{\cP}{\mathcal{P}}
\newcommand{\cM}{\mathcal{M}}
\newcommand{\cH}{\mathcal{H}}
\newcommand{\cF}{\mathcal{F}}
\newcommand{\cX}{\mathcal{X}}
\newcommand{\cE}{\mathcal{E}}

\newcommand{\rd}{\mathrm{d}}

\startlocaldefs
\theoremstyle{plain}
\newtheorem{theorem}{Theorem}
\newtheorem{lemma}{Lemma}
\newtheorem{corollary}{Corollary}
\newtheorem{proposition}{Proposition}
\newtheorem{definition}{Definition}
\newtheorem{example}{Example}
\newtheorem{remark}{Remark}

\endlocaldefs
\DeclareMathOperator*{\argmin}{\arg\min}

\begin{document}

\begin{frontmatter}
\title{Exponential Concentration for Geometric-Median-of-Means in Non-Positive Curvature Spaces}
\runtitle{Exponential Concentration on NPC Spaces}

\begin{aug}
\author[A]{\fnms{Ho} \snm{Yun}\ead[label=e1]{zrose0921@gmail.com}}
\and
\author[A]{\fnms{Byeong U.} \snm{Park}\ead[label=e2,mark]{bupark@snu.ac.kr}}
\runauthor{H. Yun and B. U. Park}
\address[A]{Department of Statistics, Seoul National University,
\printead{e1,e2}}
\end{aug}

\begin{abstract}
In Euclidean spaces, the empirical mean vector as an estimator of the population mean
is known to have polynomial concentration unless a strong tail assumption is imposed on the underlying probability measure. The idea of median-of-means tournament has been considered
as a way of overcoming the sub-optimality of the empirical mean vector.
In this paper, to address the sub-optimal performance of the empirical mean in a more general setting,
we consider general Polish spaces with a general metric, which are allowed to be non-compact and of infinite-dimension. We discuss the estimation of the associated population Fr\'echet mean,
and for this we extend the existing notion of median-of-means to this general setting.
We devise several new notions and inequalities associated with the geometry of the underlying metric, and using them
we study the concentration properties of the extended notions of median-of-means as the estimators of the population Fr\'echet mean.
We show that the new estimators achieve exponential concentration under only a second moment
condition on the underlying distribution, while the empirical Fr\'echet mean has polynomial
concentration.
We focus our study on spaces with non-positive Alexandrov curvature
since they afford slower rates of convergence than spaces with positive curvature.
We note that this is the first work that derives non-asymptotic concentration inequalities for
extended notions of the median-of-means in non-vector spaces with a general metric.
\end{abstract}


\end{frontmatter}

\section{Introduction}

The notion of a Fr\'echet mean extends the definition of mean, as a center of probability distribution, to metric space settings. Given a Borel probability measure $P$ on a metric space $(\mathcal{M},d)$ and a functional $\eta:\mathcal{M}\times\mathcal{M}\rightarrow\mathbb{R}$, the \textit{Fr\'echet mean} (or the barycenter) \cite{frechet1948elements} of $P$ is any $x^*$ such that
\begin{equation}\label{F-mean}
x^* \in \argmin_{x \in \mathcal{M}} \int_{\mathcal{M}}\eta(x,y)\,\mathrm{d} P(y).
\end{equation}
This accords with the usual definition of the Euclidean mean for $\mathcal{M}=\mathbb{R}^{D}$
when $\eta(x,y)=d(x,y)^{2}=|x-y|^2$. In this paper, we consider the estimation of the Fr\'echet mean of a heavy-tailed distribution.
Our goal is to find estimators that have better non-asymptotic accuracy than the \textit{empirical Fr\'echet mean},
\begin{equation}\label{emp-F-mean}
x_{n} \in \argmin_{x \in \mathcal{M}} \frac{1}{n}\sum_{i=1}^{n} \eta(x,X_{i})
\end{equation}
when $P$ is heavy-tailed on $\mathcal{M}$. The $x_n$ is an $M$-estimator in a broad sense. The present work is an achievement of this goal
for global {\it non-positive curvature} (NPC) spaces, also called CAT(0) or Hadamard spaces, that are of finite- or infinite-dimension.

Our coverage with NPC spaces is genuinely broad enough. It includes Hilbert spaces with Euclidean spaces as a special case, and various other types of metric spaces,
some of which are listed below.
\begin{itemize}
\item A hyperbolic space $\mathcal{H}_{D}$ has constant non-positive sectional curvature, which results in rich geometrical features due to explicit expressions for the $\log$ and $\exp$ maps.
The deviation of two geodesics in a hyperbolic space accelerates while drifting away from the origin, which allows a natural hierarchical structure in neural networks  \cite{ganea2018hyperbolic, tifrea2018poincar}.

\item The space $\mathcal{S}_{D}^{+}$ of symmetric positive definite matrices has non-constant and non-positive sectional curvature, which appears frequently in diffusion tensor imaging \cite{fillard2005extrapolation, fillard2007measuring}.
The space $\mathcal{S}_{D}^{+}$ is not only a Riemannian manifold, but also an Abelian Lie group with additional algebraic structure \cite{Arsigny2007,Lin2019,pennec2019riemannian}. Thus, additive regression modeling is allowed for random elements taking values in $\mathcal{S}_{D}^{+}$ \cite{lin2020additive}.

\item The Wasserstein space $\mathcal{P}_{2}(\mathbb{R})$ over $\mathbb{R}$ has vanishing Alexandrov curvature \cite{kloeckner2010geometric} and plays a fundamental role in optimal transport \cite{villani2009optimal}. The Wasserstein space has rich applications in modern theories,
such as change point detection \cite{horvath2021monitoring}, and Wasserstein regression \cite{chen2021wasserstein, zhang2021wasserstein, ghodrati2021distribution}.
\end{itemize}
Apart from the above-mentioned examples, there are other NPC spaces, such as phylogenetic trees \cite{ pennec2019riemannian, billera2001geometry}, that are of great importance
in applications.

A great deal of statistical inference is fundamentally based on the estimation of the Fr\'echet mean $x^*$.
While classical statistics leaned toward the asymptotic behavior of estimators, the derivation of non-asymptotic probability bounds,
called {\it concentration or tail inequalities}, has drawn increasing attention recently.
For an estimator $\hat{x}=\hat{x}(X_{1}, \dots,X_{n})$ of $x^*$,
concentration inequalities for $\hat{x}$ are given in the form of
\begin{equation}\label{radius}
\mathbb{P}\left(d(\hat{x},x^*) \leq r(n,\Delta) \right) \geq 1-\Delta,
\end{equation}
where $r(n,\Delta)$ is the radius of concentration corresponding to a tail probability level $\Delta$ whose dependence on $n$ is typically determined by the metric-entropy of $\cM$.
There have been only a few attempts to establish such concentration inequalities when $(\mathcal{M},d)$ is not a linear space, and all of them have been restricted to
the empirical Fr\'echet mean $\hat{x}=x_n$, to the best of our knowledge.
For $\cM=\mathbb{R}^D$, it is widely known that the empirical mean $x_n$ is sub-optimal
achieving only {\it polynomial concentration} for heavy-tailed $P$ in the sense that $\Delta^{-1}=f(n,r(n,\Delta))$
for some $f$ with $f(n,r)$ for fixed $n$ being a polynomial function of $r$.

A solution to alleviating the sub-optimality of the empirical mean $x_n$ is to
partition $\{X_1, \ldots, X_n\}$ into a certain number of blocks and then take a `median' of the within-block
sample means. This robustifies the empirical mean against heavy-tailed distribution
while it inherits its efficiency for light-tailed distribution. The idea
was first introduced by \cite{nemirovskij1983problem}.
When $\mathcal{M}=\mathbb{R}$, the resulting estimator, termed as {\it median-of-means},
achieves the concentration inequality \eqref{radius}
with $r(n,\Delta)=C \times n^{-1/2}\sqrt{\log(1/\Delta)}$ for some constant $C>0$ \cite{devroye2016sub}.
The one-dimensional result was extended to $\cM=\mathbb{R}^D$ by \cite{lugosi2019sub} developing the idea of `median-of-means tournament'.
The resulting estimator $\hat{x}$, also termed as median-of-means, was found to achieve a \textit{sub-Gaussian performance}:
\begin{equation}\label{sub-Gauss_conc}
\mathbb{P}\left(\|\hat{x}-x^*\| \leq C_{1} \sqrt{\frac{\text{tr}(\Sigma_{X})}{n}}+ C_{2} \sqrt{\frac{\|\Sigma_{X}\| \log(1/\Delta)}{n}} \right) \geq 1-\Delta
\end{equation}
for some constants $C_{1}, C_{2}>0$, where $\Sigma_{X}$ is the covariance matrix and $\| \cdot \|$ is the operator norm.
The concentration property at \eqref{sub-Gauss_conc} is what the empirical mean achieves when $X$ has a multivariate sub-Gaussian distribution, so the name
sub-Gaussian performance.
Both results establish {\it exponential concentration} in the sense that
$\Delta^{-1}=f(n,r(n,\Delta))$ with $f(n,r)$ for fixed $n$ being an exponential function of $r$.
There have been also proposed several other mean estimators satisfying \eqref{sub-Gauss_conc} that can be computed in linear time $(nD \log(1/\Delta))^{O(1)}$ by using the median-of-means principle,
see \cite{hopkins2020mean,cherapanamjeri2019fast, depersin2022robust, lei2020fast}, and other related works on robust mean estimation, e.g. \cite{catoni2012challenging, lugosi2021robust}.

All the aforementioned works, however,
treated Euclidean spaces for $\eta=d^{2}$ with extensive use of the associated inner product. Apart from the Euclidean cases, there are few works for infinite-dimensional $\cM$, e.g. \cite{lerasle2019monk} for a kernel-enriched domain and
\cite{minsker2015geometric} for a Banach space, both of which considered $\eta=d^{2}$.
We are also aware of \cite{hsu2016loss} that studied the case of arbitrary metric spaces.
However, the latter work does not use the geometric features of the underlying metric space
but assumes certain high-level conditions. The conditions include the existence of an estimator $\hat{x}$ and a random distance $DIST$ on $(\mathcal{M},d)$
such that $\mathbb{P} (d(\hat{x}, x^*) \leq \varepsilon) \geq 2/3$ for some $\varepsilon>0$ and
$\mathbb{P}(d(x,y)/2 \leq DIST(x,y) \leq 2d(x,y)) \geq 8/9$
for all $x,y \in \mathcal{M}$.
We highlight that the present work is the first to use the median-of-means principle without imposing strong assumptions such as in \cite{hsu2016loss}
when $\cM$ is a non-vector space. Our technical development is significantly different from the existing works in the literature.
We note that there is no distribution on non-vector spaces corresponding to Gaussian or sub-Gaussian distribution on $\mathbb{R}^D$,
neither are available the notions of trace and operator norm, so that an analogue of the sub-Gaussian performance as at \eqref{sub-Gauss_conc} for non-vector spaces does not seem to be possible.
Nevertheless, we establish for our estimators exponential concentration in the sense that the inverse of `probability regret' $\Delta$ at
\eqref{radius} is an exponential function of the radius of concentration.

In this paper, we first extend the notion of median-of-means to general metric spaces $\cM$.
Then, we address the problem of robust estimation by taking into account the {\it metric geometry} of the underlying space.
To this end, we use the {\it CN (`Courbure N\'egative' in French)}, {\it quadruple and variance inequalities},
which are not well known in statistics, instead of the inner product.
We show that, when $\cM$ is an NPC space and $\eta(x,y)=d(x,y)^{\alpha}$, the
corresponding {\it geometric-median-of-means} estimator achieves exponential concentration for all $\alpha \in (1,2]$,
under only the second moment
condition $\mathbb{E}\,d(x^*,X)^2<+\infty$.
In particular, for the treatment of the `bridging' case where $\alpha \in (1,2)$,
we introduce a further extended notion of the geometric-median-of-means, for which
we devise generalized versions of the CN and variance inequalities.
Our work is the first that provides concentration inequalities for
median-of-means type estimators
with explicit constants, when $\eta$ is not necessarily $d^{2}$ or $\cM$ is a possibly infinite-dimensional non-vector space.

We work with (possibly non-compact) NPC spaces for the geometric-median-of-means estimators since the Fr\'echet mean $x_n$ has poor performance in such spaces.
In fact, the concentration properties of $x_n$ depend heavily on the compactness and curvature of $\cM$.
For general Polish spaces, an exponential concentration inequality may be established
with $x_n$ if the space is compact \cite{ahidar2019rate}.
For non-compact geodesic spaces, however, only polynomial concentration is possible with $x_n$
unless a strong assumption on the tail of $P$ is imposed. The latter was proved
for Euclidean spaces, a special case of non-compact spaces \cite{catoni2012challenging}.
As for the curvature of the underlying space, $x_n$ has a poorer rate of convergence for $\cM$ with non-positive
curvature than with positive curvature (\cref{Sec_Empirical means,Subsection 3.2.3}). Curvature and compactness are related in the case where $\cM$ is a Riemannian manifold. The Bonnet-Myers theorem states that, if the sectional curvature of a Riemannian manifold is bounded
from below by $\kappa>0$, then ${\rm diam}(\cM) \le \pi/\sqrt{\kappa}$ so that it is compact.
To complement the existing works for $x_n$, we demonstrate the polynomial concentration of $x_n$, as well, for general Polish spaces in \cref{Sec_Empirical means},
and for NPC spaces as a specialization of the latter in \cref{Section 4}.
We note that there have been few works on non-asymptotic theory of $x_n$ for non-Euclidean $\cM$,
although its asymptotic theory has been widely studied \cite{bhattacharya2003large, bhattacharya2005large,le2017existence,
sturm2003probability}.
The work in \cref{Sec_Empirical means} for the empirical Fr\'echet mean $x_n$ paves our way
for developing the main results in \cref{Sec_MoM} for the geometric-median-of-means estimators.

Our treatment of NPC spaces relies on the metric geometry of the underlying space $\cM$, rather than on the differential geometry of $\cM$.
Consequently, the radius of concentration $r(n,\Delta)$ in the exponential inequalities in \cref{Sec_MoM} does not involve any term related
to the structure of the tangential vector space of $\cM$, which corresponds to $\Sigma_X$ in Lugosi \cite{lugosi2019sub} when $\cM=\mathbb{R}^D$.
We find that assuming $\mathbb{E}\,d(x^*,X)^{2}<+\infty$ is enough
to deduce the exponential concentration. The flexibility inherent in our framework thus allows our work to serve
as the basic constituent for a wide range of principal methods for non-Euclidean data.
In particular, the theoretical development achieved in this paper may be adapted to
the robustification of various recent Fr\'echet regression techniques
\cite{lin2020additive, chen2021wasserstein, zhang2021wasserstein, ghodrati2021distribution, petersen2019frechet}.

\section{Assumptions}\label{assumptions}

In this section, we present the main structures of the underlying metric, where we base our theory,
and key assumptions on the entropy of the underlying space.
The validity of the assumptions will be discussed in \cref{Section 4}.

Let $(\mathcal{M},d)$ be a separable and complete metric space (Polish space).
Consider the set of all probability measures on $\cM$ denoted by
$\cP(\cM)$. Let $P$ be a probability measure with finite second moment, i.e.
\begin{equation*}
P \in \mathcal{P}_2(\mathcal{M}):= \Big\{P \in \mathcal{P}(\cM): \int_{\mathcal{M}}d(x,y)^{2} \,\rd P(y)<+\infty\,\mbox{ for some }\, x \in  \mathcal{M}\Big\}.
\end{equation*}
We note that, if $\int_{\mathcal{M}}d(x,y)^{2} \,\rd P(y)<+\infty$ for some $x \in \cM$, then it holds for all $x\in \cM$.
Let $\eta:\mathcal{M}\times\mathcal{M} \rightarrow \mathbb{R}$ be a measurable function.
Throughout this paper, we assume that there exists $x^* \in \mathcal{M}$ for which \eqref{F-mean} holds and let $X_{1}, X_{2}, \dots, X_{n}$ be the i.i.d. observations of a random element $X$ governed by a probability measure $P$, and $P_{n}$ be its empirical probability measure.
Then, the empirical Fr\'echet mean $x_n$ at \eqref{emp-F-mean} can be written as
\begin{equation*}
x_{n} \in \argmin_{x \in \mathcal{M}} \int_{\mathcal{M}} \eta(x,y)\,\mathrm{d} P_{n}(y).
\end{equation*}

To analyze the deviation of $x_{n}$ from $x^*$ by making use of the difference of their $\eta$-functional values,
we introduce two assumptions, the first on $\eta:\mathcal{M}\times\mathcal{M} \rightarrow \mathbb{R}$ and the second on $P\in \mathcal{P}_2(\mathcal{M})$:
\begin{itemize}\setlength{\itemindent}{1em}
\item[(A1)] {\it Quadruple inequality}:  There is a nonnegative function
$l:\mathcal{M}\times\mathcal{M}\rightarrow [0,+\infty)$, called \textit{growth function},
such that, for any $y,z,p,q \in \mathcal{M}$,
\begin{equation*}
\begin{split}
& l(y,z)=0 \;\Leftrightarrow \;y=z, \\
&\eta(y,p)-\eta(y,q)-\eta(z,p)+\eta(z,q) \leq 2l(y,z)\cdot d(p,q).
\end{split}
\end{equation*}
\item[(A2)] {\it Variance inequality}: There exist constants $K>0$ and  $\beta \in(0,2)$ such
that, for all $x \in \mathcal{M}$,
\begin{equation*}
l\left(x, x^{*}\right)^{2} \leq K\left(\int_{\mathcal{M}}\left(\eta(x, y)-\eta\left(x^{*}, y\right)\right) \mathrm{d} P(y)\right)^{\beta}.
\end{equation*}
\end{itemize}
We note that (A1) and (A2) together imply the {\it uniqueness} of the Fr\'echet mean $x^*$.

\begin{example}\label{hilb_example1}
{\rm Consider the case where $\cM$ is a Hilbert space $\cH$ with an inner product $\langle \cdot, \cdot\rangle$ and $d(x,y)=\|x-y\|$ for the induced norm $\|\cdot\|$ of $\langle \cdot, \cdot\rangle$.
Let $\eta=d^2$. If $X$ has finite second moment, i.e. $\mathbb{E}\,d(x^*,X)^2<+\infty$,
then $x^*=\mathbb{E}X$ is the {\it unique} barycenter of $X$ in the sense of Bochner integration. Also, it holds that
\begin{equation*}
\begin{split}
&\eta(y,p)-\eta(y,q)-\eta(z,p)+\eta(z,q) \\
&\qquad =(2\langle y-q,q-p \rangle+ \|q-p\|^{2})-(2\langle z-q,q-p \rangle+ \|q-p\|^{2}) \\
&\qquad =2\langle y-z,q-p \rangle\\
&\qquad \leq 2\|y-z\| \cdot \|p-q\|.
\end{split}
\end{equation*}
Thus, (A1) holds with $l=d$, and (A2) does with equality holding
always for all $x \in \cM$ with} $K=\beta=1$:
\begin{align*}
\mathbb{E}\left(\eta(x, X)-\eta\left(x^{*}, X\right)\right)=\mathbb{E}\left(2\langle x^{*}-X,x-x^{*} \rangle+ \|x-x^{*}\|^{2}\right)
=\|x-x^{*}\|^{2}. \quad \mbox{\qed}
\end{align*}
\end{example}

For curved spaces, the inequality in (A2) may be satisfied, but with equality
not holding always for all $x\in\cM$ in general, contrary to the Hilbertian
case. Moreover, both $x_{n}$ and $x^*$ do not have a closed form expression for curved metric spaces although $x_{n}$ has for Hilbert spaces. Therefore, in order to derive a concentration inequality for $x_{n}$,
we need an inequality that gives an upper bound to the discrepancy $l(x_{n}, x^{*})$ between $x_{n}$ and $x^*$.
The variance inequality (A2) implies that $l(x_{n}, x^{*})$ can be controlled by the positive function $\eta(x_{n}, \cdot)-\eta(x^{*}, \cdot)$,
called the \textit{empirical excess risk} of $\eta$:
\begin{equation}\label{empi_var_ineq}
    l\left(x_{n}, x^{*}\right)^{2} \leq K\left(\int_{\mathcal{M}}\left(\eta(x_{n}, y)-\eta\left(x^{*}, y\right)\right) \mathrm{d} P(y)\right)^{\beta}.
\end{equation}
For the usual choice $\eta=d^{2}$, it turns out that (A1) and (A2) hold with $l=d, K=\beta=1$ for general
NPC spaces $\mathcal{M}$, see \cref{Subsection 4.1} for details.

Bounding the right hand side of \eqref{empi_var_ineq} with a high probability depends on the geometric properties of the class of functions
$\eta(x, \cdot)-\eta(x^{*}, \cdot)$ for $x \in \cM$. It turns out that the dependence is through the {\it centered functional} $\eta_c$
defined by $\eta_c(x,\cdot)=\eta(x,\cdot)- \int_{\mathcal{M}} \eta(x,y) \rd P(y)$. Put $f_\eta(x,\cdot)=\eta_c(x,\cdot)-\eta_c(x^*,\cdot)$, $x \in \cM$.

\begin{definition}\label{def1}
For $\delta \geq 0$,
\begin{equation*}
\begin{split}
    &\cM_\eta(\delta)=\left\{x \in \mathcal{M}: \int_{\mathcal{M}}\left(\eta(x, y)-\eta\left(x^{*}, y\right)\right) \mathrm{d} P(y) \leq \delta\right\}, \\
    &\mathcal{F}_\eta(\delta)=\{f_\eta(x,\cdot): x \in \cM_\eta(\delta)\},\\
    &\sigma_\eta^{2}(\delta)=\sup \left\{\int_{\mathcal{M}} f_\eta(x,y)^{2} \,\rd P(y) : x \in \cM_\eta(\delta)\right\}.
\end{split}
\end{equation*}
\end{definition}

\begin{example}\label{hilb_example2}
{\rm Consider the $\eta$ and $X$ in \cref{hilb_example1}.
Let $\Sigma_{X}:\mathcal{H} \times \mathcal{H} \rightarrow \mathbb{R}$ be the covariance operator of $X$ defined by
$\Sigma_{X}(x,y)=\mathbb{E}\,\left(\langle x,X-x^* \rangle \langle y,X-x^* \rangle\right)$ and $\lambda_{max}$ be its largest eigenvalue.
From \cref{hilb_example1}, it is straightforward to see that $\cM_\eta(\delta)=B(x^*,\sqrt{\delta})$ and $\mathbb{E} \,\eta(x,X)=\operatorname{tr}(\Sigma_{X})+\|x-x^*\|^{2}$,
where $B(x,r)$ denotes the ball centered at $x$ with radius $r$, $\operatorname{tr}(\Sigma_{X})=\sum_{k} \Sigma_{X}(e_k,e_k)$ and
$\{e_k: k \ge 1\}$ is an arbitrary orthonormal basis of $\mathcal{H}$.
Let $\|\cdot\|_{2,P}$ be defined by $\|f\|_{2,P}^2=\mathbb{E} \,f(X)^2$. Then,
\begin{equation*}
\begin{split}
\eta_c(x,y)&=\|x-y\|^{2}-\|x-x^*\|^{2}-\operatorname{tr}(\Sigma_{X}),\\
f_\eta(x,y)&=2\langle x-x^{*},x^{*}-y \rangle, \\
\|f_\eta(x,\cdot)-f_\eta(y,\cdot)\|^{2}_{2,P}&= 4\,\mathbb{E} \,\left(\langle x-y,X-x^{*} \rangle^2\right)=4 \,\Sigma_{X}(x-y, x-y).
\end{split}
\end{equation*}
Note that $f_\eta(x,\cdot):\mathcal{H} \rightarrow \mathbb{R}$ is an affine function and $f_\eta(x^*,\cdot)\equiv 0\equiv f_\eta(\cdot,x^*)$. Also, from the Cauchy-Schwarz inequality, we have
\begin{equation*}
\begin{split}
   \sigma_\eta^{2}(\delta)
   &=\sup \left\{4 \mathbb{E}(\langle x-x^{*}, X-x^{*} \rangle^{2}) : x \in B(x^*,\sqrt{\delta})\right\} \\
   &=\sup \left\{4 \Sigma_{X}(x-x^{*},x-x^{*}) : x \in B(x^*,\sqrt{\delta})\right\} \\
   &=4 \delta \cdot \lambda_{max}. \quad \mbox{\qed}
\end{split}
\end{equation*}}
\end{example}

Under the assumptions (A1) and (A2), it holds that
\begin{align}\label{env-bound}\begin{split}
\sup_{x \in \cM_\eta(\delta)}f_\eta(x,y)
&=\sup_{x \in \cM_\eta(\delta)}
\int_{\mathcal{M}}\left(\eta(x,y)-\eta(x^*,y)-\eta(x,z)+\eta(x^*,z)\right) \,\rd P(z) \\
&\leq 2\sup_{x \in \cM_\eta(\delta)}\int_{\mathcal{M}}l(x,x^*)d(y,z) \,\rd P(z) \\
&\leq 2\sqrt{K \delta^{\beta}} \int_{\mathcal{M}}d(y,z) \,\rd P(z)=:H_{\delta,\eta}(y).
\end{split}
\end{align}
By definition $H_{\delta,\eta}$ {\it envelops} the class $\mathcal{F}_\eta(\delta)$ of functions under the assumptions (A1) and (A2).
Let $\|\cdot\|_{2,P_{n}}$ be defined by
\[
\|f\|_{2,P_n}^2=n^{-1}\sum_{i=1}^n f(X_i)^2, \quad f:\mathcal{M} \rightarrow \mathbb{R}.
\]
Note that $\|\cdot\|_{2,P_{n}}$ is a pseudo metric.
To analyze high probability concentration, toward zero, of the right hand side of \eqref{empi_var_ineq}, we consider the following assumption on the $\|\cdot\|_{2,P_n}$-metric entropy
of $\cM$.
For a totally bounded subset $\mathcal{S}$ of a metric space $(\cF,d)$, we let
$N(\tau,\mathcal{S},d)$ denote the minimal number of balls with radius $\tau$ that cover $\mathcal{S}$, and call it $\tau$-covering number.

\begin{itemize}\setlength{\itemindent}{1em}
\item[(B1)] {\it Finite-dimensional $\cM$}: There are some constants $A, D>0$ such that, for any $\delta>0$ and $n \in \mathbb{N}$,
\[
N\left(\tau\|H_{\delta,\eta}\|_{2,P_{n}},\mathcal{F}_\eta(\delta), \|\cdot\|_{2,P_{n}}) \right) \leq\left(\frac{A}{\tau}\right)^{D}, \quad 0<\tau \leq 1,
\]
\end{itemize}

The constant $D$ in the assumption (B1) is related to the index of VC(Vapnik-\u{C}ervonenkis)-type class of functions, which appears frequently in M-estimation.
According to the common definition \cite{gine2021mathematical}, $\mathcal{F}_\eta(\delta)$ is of VC-type with respect to $H_{\delta,\eta}$ if
\begin{align}\label{VC-type}
\sup_{Q \in \mathcal{P}(\mathcal{M})}N\left(\tau\|H_{\delta,\eta}\|_{2,Q},\mathcal{F}_\eta(\delta), \|\cdot\|_{2,Q}) \right) \leq\left(\frac{A}{\tau}\right)^{D_{\rm vc}}
\end{align}
for some constants $A, D_{\rm vc}>0$. The constant $D_{\rm vc}$, termed as {\it VC index}, may not equal the dimension of $\cM$ in general, but is usually larger, and
\eqref{VC-type} implies (B1) with $D=D_{\rm vc}$, the latter being what we actually need in our framework.
Because of the implication, $\mathcal{F}_\eta(\delta)$ with (B1) may be regarded as a \textit{weak} VC-type class
of functions, and $D$ as a \textit{weak} VC index. In \cref{prop_npc_cover} given later in
\cref{Section 4} we show that (B1) holds with $D=\operatorname{dim}(\mathcal{M})$ in the case where
$\cM$ is an NPC space with $\operatorname{dim}(\mathcal{M})<+\infty$ and $\eta=d^2$.

For infinite-dimensional scenarios, we make the following assumption on the geometric complexity of $\cM$.

\begin{itemize}\setlength{\itemindent}{1em}
\item[(B2)] {\it Infinite-dimensional $\cM$}:
There are some constants $A, \zeta>0$ such that, for any $\delta>0$ and $n \in \mathbb{N}$,
\begin{equation*}
\log N\left(\tau\|H_{\delta,\eta}\|_{2,P_{n}},\mathcal{F}_\eta(\delta), \|\cdot\|_{2,P_{n}}) \right) \leq \frac{A}{\tau^{2\zeta}}, \quad 0<\tau \leq 1.
\end{equation*}
\end{itemize}

The constant $\zeta$ describes how quickly the covering number grows as $\tau$ decreases. For probability measures $P$ with non-compact support,
the complexity constant depends largely on the curvature of $\mathcal{M}$.
Here and throughout the paper, `curvature' means sectional curvature for Riemannian manifolds, and
Alexandrov curvature for general metric spaces.
When $\eta=d^{2}$, we get that $\zeta=1$ for Hilbert spaces $\cM$,
$\zeta \leq 1$ for geodesic spaces with positive curvature, and
$\zeta \geq 1$ for geodesic spaces with non-positive curvature, see \cref{Subsection 3.2.3}.
Based on this, we call $\zeta$ the \textit{curvature complexity} of $\cM$.


\section{Empirical Fr\'echet means}\label{Sec_Empirical means}

In  this section, we present two theorems that establish polynomial concentration for empirical Fr\'echet means
under the assumptions (A1), (A2), (B1) and (B2) in the case where $\cM$ is a general Polish space.
The theorems are used in developing exponential concentration for geometric-median-of-means estimators to be introduced in \cref{Sec_MoM}.
Throughout this section, we assume that $P$ has finite second moment, i.e., $\sigma_{X}^2:=\mathbb{E}\,d(x^*,X)^{2}<+\infty$.

\begin{theorem}\label{thm_fin_conv}
Assume (A1), (A2) and (B1), and let $(K,\beta)$ and $(A, D)$ be the constant pairs that appear in (A2) and (B1), respectively.
Then, for all $n\in \mathbb{N}$ and $\Delta \in (0,1)$,
\begin{align*}
l(x_{n},x^*)
\leq C_\Delta\cdot \left( \frac{\sigma_{X}}{\sqrt{n}} \right)^{\frac{\beta}{2-\beta}}
\end{align*}
with probability at least $1-\Delta$, where $C_\Delta$ is given by
\begin{equation*}
C_\Delta= K^{\frac{1}{2-\beta}}\left\{32 \left(24\sqrt{AD}+\sqrt{\frac{2}{\Delta}} \right) \right\}^{\frac{\beta}{2-\beta}}.
\end{equation*}
\end{theorem}

In the case where $\cM$  is an NPC space to be introduced in the next section, choosing $\eta=d^{2}$ gives $l=d$ and $K=\beta=1$, see \cref{Subsection 4.1}.
In this case, \cref{thm_fin_conv} provides an upper bound of order $\sigma_{X}/\sqrt{n \Delta}$ for $d(x_{n},x^*)$.
Note that, in the trivial case where $\cM=\mbR^D$ with $d(x,y)=|x-y|$, an application of the Chebyshev inequality gives
\[
\mathbb{P}\left(|x_n-x^*|\le \frac{\sigma_X}{\sqrt{n\Delta}}\right) \ge 1-\Delta.
\]
Here and throughout this paper, $|\cdot|$ denotes the Euclidean norm. The extra factor $C_\Delta$ in \cref{thm_fin_conv} is a price we pay for the complexity of $\cM$ to deal with general metric spaces.
The following theorem is for infinite-dimensional scenarios with the assumption (B2).

\begin{theorem}\label{thm_inf_conv}
Assume (A1), (A2) and (B2), and let $(K,\beta)$ and $(A,\zeta)$ be the constant pairs that appear in (A2) and (B2), respectively.
Then, there is a universal constant $C_{A,\zeta}$ depending only on $A>0$ and $\zeta >0$ such that, for all $n\in \mathbb{N}$ and $\Delta \in (0,1)$,
\begin{align*}
l(x_{n},x^*)
\leq
\begin{cases}
\displaystyle K^{\frac{1}{2-\beta}}\left(C_{A,\zeta} \cdot \frac{1}{n^{1/2}} \cdot \frac{\sigma_{X}}{\sqrt{\Delta}} \right)^{\frac{\beta}{2-\beta}}, \quad &\mbox{if}\;\;0<\zeta <1\\ 
\quad \\[-6mm]
\displaystyle K^{\frac{1}{2-\beta}}\left(C_{A,1} \cdot \frac{\log n}{n^{1/2}} \cdot \frac{\sigma_{X}}{\sqrt{\Delta}} \right)^{\frac{\beta}{2-\beta}}, \quad &\text{if}\;\;\zeta = 1\\ 
\quad \\[-6mm]
\displaystyle K^{\frac{1}{2-\beta}}\left(C_{A,\zeta}\cdot \frac{1}{n^{1/2\zeta}} \cdot \frac{\sigma_{X}}{\sqrt{\Delta}}\right)^{\frac{\beta}{2-\beta}}, \quad &\text{if}\;\;\zeta>1\\ 
\end{cases}
\end{align*}
with probability at least $1-\Delta$.
\end{theorem}

An explicit form of the constant $C_{A,\zeta}$ in \cref{thm_inf_conv} may be found in the proof of the theorem in \cref{proofs-Sec3}. 
The theorem demonstrates that the consistency of the empirical Fr\'echet mean $x_{n}$ continues to hold for infinite-dimensional $(\mathcal{M},d)$, but with slower rates of convergence to $x^*$ for increasing $n$ when $\zeta \ge 1$,
compared to the finite-dimensional case in \cref{thm_fin_conv}. It shows that, for infinite-dimensional geodesic spaces $\cM$, decreasing the curvature of $\cM$ results in slowing down the rate of convergence of $x_{n}$ to $x^*$
since the curvature complexity {\it $\zeta$ gets larger as the curvature decreases}.
This implies that the rate is slower for $\cM$ with non-positive curvature than with positive curvature.
We note that, for the finite-dimensional case, the rate of convergence of $x_{n}$ does not depend on the curvature, as is shown in
\cref{thm_fin_conv}. The constant $A$ in $C_\Delta$, however, gets larger as the curvature of $\mathcal{M}$ decreases
in the case where $\cM$ is a Riemannian manifold and $\eta=d^{2}$, see \cref{Subsection 3.2.3}.

\cref{thm_fin_conv,thm_inf_conv} reveal that the empirical Fr\'echet mean achieves only polynomial concentration speeds.
In \cref{Sec_MoM} we discuss in depth alternative estimators that have exponential speeds, basically replacing $1/\Delta$ by $\log(1/\Delta)$ in the concentration inequalities.

\section{Consideration of assumptions}\label{Section 4}

In this section, we discuss the validity of the assumptions (A1), (A2), (B1) and (B2) for non-positive curvature (NPC) spaces.
We also derive generalized versions of the CN and variance inequalities.

\begin{definition}\label{NPC-sp}
A Polish space $(\mathcal{M},d)$ is called an (global) NPC space if for any $x_{0},x_{1} \in \mathcal{M}$, there exists $y  \in \mathcal{M}$ such that
\begin{equation*}
d(z,y)^{2} \leq \frac{1}{2}d(z,x_0)^{2}+\frac{1}{2}d(z,x_1)^{2}-\frac{1}{4}d(x_0,x_1)^{2}, \quad z \in \mathcal{M}.
\end{equation*}
\end{definition}

\begin{example}\label{hilb_example3}
{\rm Any Hilbert space $(\mathcal{H},\langle\cdot,\cdot\rangle)$ is an NPC space: for any $x_{0},x_{1},z \in \mathcal{M}$
\begin{equation*}
\begin{split}
\frac{1}{2}d(z,x_0)^{2}+\frac{1}{2}d(z,x_1)^{2}-\frac{1}{4}d(x_0,x_1)^{2}
&=\frac{1}{4} \left(2\|z-x_0\|^{2}+2\|z-x_1\|^{2}-\|(z-x_0)-(z-x_1)\|^{2} \right) \\
&=\frac{1}{4} \, \|(z-x_0)+(z-x_1)\|^{2}\\
&= d\left(z,\frac{x_0+x_1}{2}\right)^{2}. \quad \mbox{\qed}
\end{split}
\end{equation*}}
\end{example}

Throughout this section, $\cM$ is an NPC space. Also, when there is no confusion, with an abuse of terminology, `Riemannian manifold' means a smooth, complete and connected finite-dimensional Riemannian manifold.
By the Hopf-Rinow Theorem, such a Riemannian manifold is geodesically complete.

\subsection{Common choice \texorpdfstring{$\eta=d^{2}$}{: squared distance}}\label{Subsection 4.1}

Let us first discuss some properties of NPC spaces when $\eta(x,y)=d(x,y)^{2}$.
The geometry of metric measure spaces with non-positive curvature is mainly developed by Sturm \cite{sturm2003probability}.
Note that the existence and uniqueness of the Fr\'echet mean for any probability measure are guaranteed for such spaces.

We have seen in \cref{hilb_example1} that, for Hilbert spaces, the inner product structure allows us to easily verify (A1) and the equality in (A2)
with $l=d$, $K=\beta=1$.
For curved spaces, however, $d(x,y)^{2}-d(x^*,y)^{2}$ cannot be expressed nicely, thus our assumptions (A1) and (A2) may not be easy to check.
For example, for Riemannian manifolds $\mathcal{M}$, the relationship between the embedded distance $\|\log_{p}x-\log_{p}y\|$  for $p,x,y \in \mathcal{M}$
and the original distance $d(x,y)$ depends considerably on the curvature, see \cref{remark-cond-AB} below.
Nevertheless, using the fact that the geodesic deviation accelerates as two geodesics move further away from the origin, one may prove
the following inequalities for global NPC spaces $\mathcal{M}$, see \cite{sturm2003probability}
for details.

\begin{itemize}
\item[] {\it CN inequality}: For any $y \in \mathcal{M}$ and for any geodesic $\gamma:[0,1]\rightarrow \mathcal{M}$,
\[
d(\gamma_{t},y)^{2} \leq (1-t) d(\gamma_{0},y)^{2}+t\, d(\gamma_{1},y)^{2} -t(1-t) d(\gamma_{0},\gamma_{1})^{2}, \quad t \in [0,1].
\]
\item[] {\it Quadruple inequality}: For any $y,z,p,q \in \mathcal{M}$,
\[
d(y,p)^{2}-d(z,p)^{2}-d(y,q)^{2}+d(z,q)^{2} \leq 2d(y,z)d(p,q).
\]
\item[] {\it Variance inequality}: For any $x \in \mathcal{M}$ and for any $P \in \mathcal{P}_{2}(\mathcal{M})$,
\[
d\left(x, x^{*}\right)^{2} \leq \int_{\cM}\left(d(x, y)^{2}-d\left(x^{*}, y\right)^{2}\right) \mathrm{d} P(y).
\]
\end{itemize}
Here, `CN' stands for Courbure N\'egative in French. Therefore, not only for Hilbert spaces but also for NPC spaces,
our assumptions (A1) and (A2) are satisfied with $K=\beta=1$, $l=d$ for the usual choice $\eta(x,y)=d(x,y)^{2}$.

\begin{remark}\label{remark-cond-AB}
We note that $\eta=d^{2}$ satisfies the Hamilton-Jacobi equation, see (14.29) in \cite{villani2009optimal},
and the homogeneous Taylor polynomial of order 4 for $\eta$ gives the following formula: for any $p \in \mathcal{M}$ and $v,w \in T_{p}\mathcal{M}$,
\begin{equation*}
     d\left(\exp_{p}(tv), \exp_{p}(tw)\right)^{2}=\|v-w\|^{2} \cdot t^{2}-\frac{1}{3} \operatorname{Riem}(v, w, w, v) \cdot t^{4}+O(t^{5}),
\end{equation*}
where `$\operatorname{Riem}$' stands for the Riemannian curvature tensor.
\end{remark}

\subsection{Cases with \texorpdfstring{$\eta=d^\alpha$}{the power transform metric}}\label{Subsection 4.2}

Here, we consider the choice $\eta=d^{\alpha}$, or equivalently $\eta=d_\alpha^2$ with $d_\alpha=d^{\alpha/2}$, for $\alpha \in (1,2]$.
We note that the Fr\'echet mean $x^*$ corresponding to $\alpha=1$ is analogous to the conventional median for $\mathcal{M}=\mathbb{R}$, thus is often called {\it Fr\'echet median}.
We exclude the case $\alpha=1$ in our discussion, however, for the reason to be given shortly. We also note that $d_\alpha$ is a metric for $\alpha \in (1,2]$,
and is often called {\it power transform metric}. The associated Fr\'echet mean is called \textit{$\alpha$-power Fr\'echet mean}.
With a slight abuse of notation we continue to denote it by $x^*$ throughout this paper.

\cref{Fermatfigure} illustrates the $\alpha$-power Fr\'echet means for several $\alpha \in [1,2]$ when $\cM=\mathbb{R}^2$, $d(x,y)=|x-y|$
and $P$ has the equal probability mass $1/3$ at three points $a_1=(0,h), \, a_2=(-\sqrt{3},0), \, a_3=(\sqrt{3},0)$.
The right panel depicts $t$ in $x^*=(0,t)$ as a function of $h$.
For $\alpha=2$, $x^*=(a_1+a_2+a_3)/3=(0,h/3)$ becomes most sensitive to the change of $a_1=(0,h)$ from a certain point on the scale of $h$.
For $\alpha=1$, $x^*=\argmin_{x \in \mathbb{R}^{2}} \overline{x a_{1}}+\overline{x a_{2}}+\overline{x a_{3}}$, known as the \textit{Fermat point}, is invariant
for $h\ge 1$. As the cases $\alpha=1.1$ and $\alpha=1.5$ demonstrate, $x^*$ for $\alpha \in (1,2)$ is resistent to outlying $a_1=(0,h)$ to a certain extent depending on $\alpha$:
the smaller $\alpha$ is, the more it resists.

\cref{Fermatfigure} also indicates that all $\alpha$-power Fr\'echet means for different values of $\alpha$ meet at $(0,1)$ when $a_1=(0,3)$. This is not a coincidence.
\cref{prop_power_mean_polygon} in the Supplementary Material shows
that, if the underlying probability measure $P$ is invariant under rotation around a point $z$,
then $z$ is the unique $\alpha$-power Fr\'echet mean for all $\alpha \ge 1$.

\begin{figure}
\centering
\includegraphics[width=14cm,height=6cm]{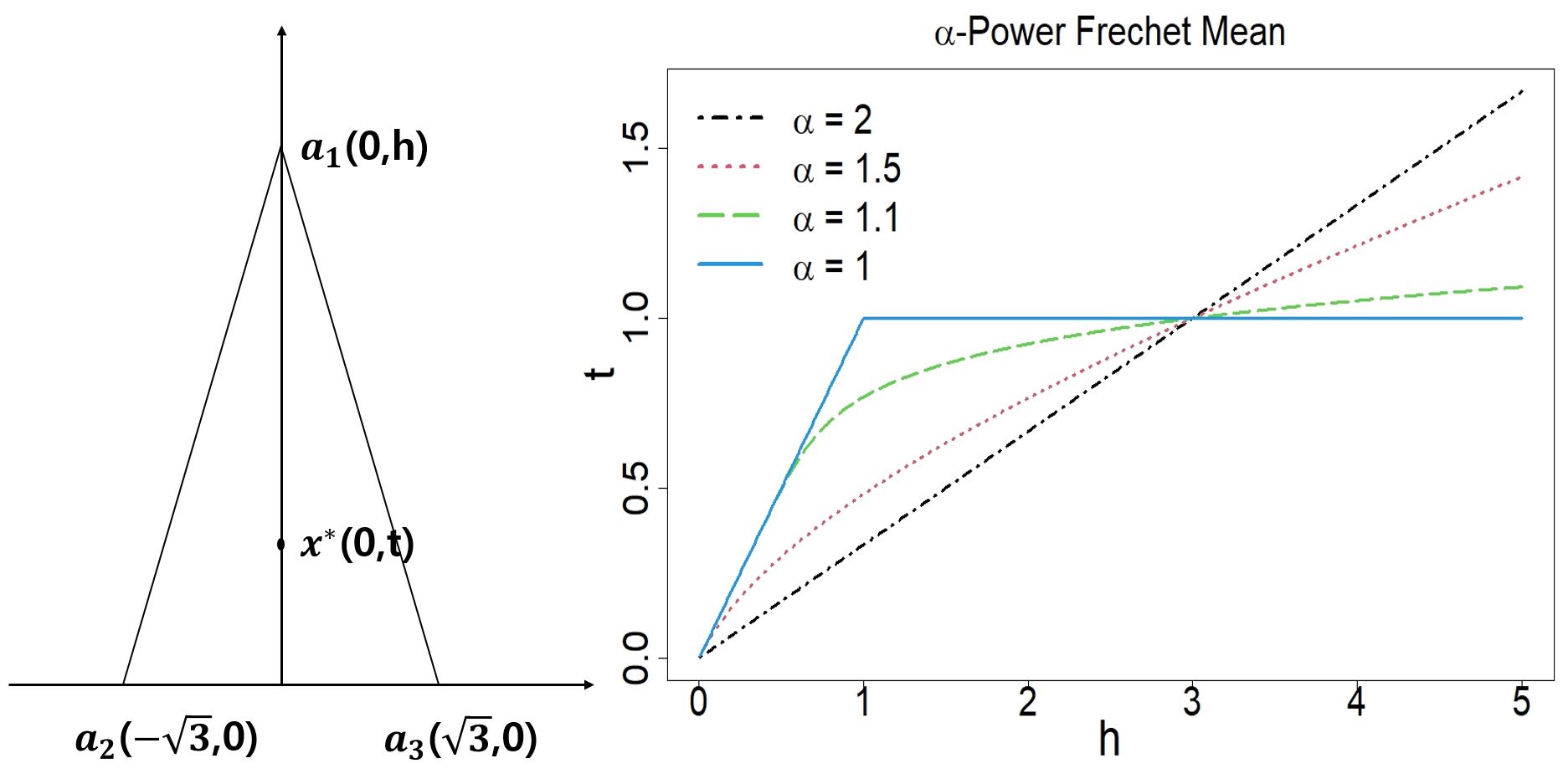}
\caption{The left panel depicts the positions of the $\alpha$-power Fr\'echet mean $x^*$ and the three points $a_1, \, a_2, \, a_3$ having equal mass.
The right panel shows the change of $x^*$ as $a_1$ moves with $a_2$ and $a_3$ staying fixed, for $\alpha=1/1.1/1.5/2$ (solid/dashed/dotted/dot-dashed).}
\label{Fermatfigure}
\end{figure}

The rates of convergence for $\alpha$-power Fr\'echet means are studied for NPC spaces with $\alpha \in [1,2]$ in Sch\"otz \cite{schotz2019convergence}.
In the latter work it is proved that the assumption (A1) holds with $l(\cdot,\cdot)=\alpha2^{-\alpha+1}d(\cdot,\cdot)^{\alpha-1}$: for any $y,z,p,q \in \mathcal{M}$,
\begin{equation}\label{power-quad-ineq}
d(y,p)^{\alpha}-d(z,p)^{\alpha}-d(y,q)^{\alpha}+d(z,q)^{\alpha} \leq \alpha2^{-\alpha+2}d(y,z)^{\alpha-1}d(p,q), \quad \alpha \in [1,2].
\end{equation}
Moreover, according to Appendix E in Sch\"otz \cite{schotz2019convergence},
no growth function satisfying (A1) exists for $\alpha>2$ and $0 < \alpha <1$.
For $\alpha=1$, (\ref{power-quad-ineq}) implies (A1) with the growth function $l(y,z)=I(y \neq z)$, but with this the assumption (A2) makes no sense,
so that \cref{thm_fin_conv,thm_inf_conv} are not meaningful for $\eta=d$.
For the case where $\alpha=1$, Ba\u c\'ak \cite{bacak2014computing} provided some results analogous to \cref{thm_fin_conv,thm_inf_conv}. Ba\u c\'ak \cite{bacak2014convex, bacak2018old} also introduced stochastic proximal point algorithms (PPA) to compute Fr\'echet medians in NPC spaces.

In the next two propositions we derive generalized CN and variance inequalities for $\alpha \in (1,2]$. Thus, the theorems
in \cref{Sec_Empirical means} remain valid for $\alpha$-power Fr\'echet means as well.

\begin{proposition}[Power transform CN inequality]\label{thm_gen_CN}
Let $\gamma:[0,1] \rightarrow \mathcal{M}$ be a geodesic and $\alpha \in [1,2]$.
Then, it holds that, for any $\delta \geq 0$, $t \in [0,1]$ and $z \in \mathcal{M}$,
\begin{align*}
d(\gamma_{t},z)^{\alpha} &\le (1+\delta)^{1-\alpha/2} \left[(1-t)^{\alpha/2} d(\gamma_{0},z)^{\alpha}+t^{\alpha/2} d(\gamma_{1},z)^{\alpha} \right]\\
&\hspace{2cm} - \delta^{1-\alpha/2} \left[t(1-t)d(\gamma_{0}, \gamma_{1})^{2}\right]^{\alpha/2}.
\end{align*}
\end{proposition}

Our result in \cref{thm_gen_CN} reduces to the CN inequality in
\cref{Subsection 4.1} when $\alpha=2$. It is believed to be a sharp generalization since it is derived from the CN inequality in \cref{Subsection 4.1}
and a version of H\"older's inequality, both of which are sharp.
When given three points $x,y,z \in \mathcal{M}$, \cref{thm_gen_CN} enables us to get an upper bound for the power transform metric $\eta(\cdot, z)=d(\cdot, z)^{\alpha}$
along the geodesic from $x$ to $y$, which does not seem to be feasible for general $\eta$.
We will illustrate how to use this inequality in a concrete way in the proof of the following proposition, and also in the proofs of the concentration inequalities given in
\cref{thm_fin_power_mom,thm_inf_power_mom} later in \cref{Section 5.2}.

To state the second proposition, for $\alpha>0$ we let
\[
\mathcal{P}_\alpha(\mathcal{M}):= \Big\{P \in \mathcal{P}: \int_{\mathcal{M}}d(x,y)^\alpha \rd P(y)<+\infty\,\mbox{ for some }\, x \in  \mathcal{M}\Big\}.
\]
For $P \in \mathcal{P}_\alpha(\mathcal{M})$, define $F_{\alpha}(\cdot)= \int_{\mathcal{M}} d(\cdot,y)^{\alpha}\mathrm{d} P(y)$ and
\begin{equation*}
b_{\alpha}(x)=\sup_{t \in (0,1]} \frac{F_{\alpha}(\gamma^{x}_{t})-\left\{t^{\alpha/2}+(1-t)^{\alpha/2}\right\} F_{\alpha}(x^*)}{t^{\alpha/2} d(x,x^*)^{\alpha}},
\quad x \in \mathcal{M}\setminus\{x^*\},
\end{equation*}
where $\gamma^{x}:[0,1] \rightarrow \mathcal{M}$ is the geodesic from $x^*$ to $x$.

\begin{proposition}[Power transform variance inequality]\label{prop_gen_var_ineq}
Let $\alpha \in [1,2]$ and $P \in \mathcal{P}_\alpha(\mathcal{M})$. If $b_{\alpha}(x)>0$, then
\begin{equation*}
d\left(x, x^{*}\right)^{\alpha} \leq \frac{1}{b_{\alpha}(x)} \int_{\cM}\left(d(x, y)^{\alpha}-d\left(x^{*}, y\right)^{\alpha}\right) \mathrm{d} P(y),
\quad x \in \mathcal{M}\setminus\{x^*\}.
\end{equation*}
Therefore, if $B_{\alpha}:=\inf_{x \in \mathcal{M}\setminus\{x^*\}} b_{\alpha}(x)>0$, then for any $x \in \mathcal{M}$,
\begin{equation*}
d\left(x, x^{*}\right)^{\alpha} \leq \frac{1}{B_{\alpha}} \int_{\cM}\left(d(x, y)^{\alpha}-d\left(x^{*}, y\right)^{\alpha} \right)\mathrm{d} P(y).
\end{equation*}
\end{proposition}

\cref{prop_gen_var_ineq} tells that, in order to establish the power transform variance inequality,
it suffices to check that, for all $x \in \mathcal{M}\setminus\{x^*\}$, $F_{\alpha}(\gamma^{x}_{t})$ gets apart from $(t^{\alpha/2}+(1-t)^{\alpha/2}) F_{\alpha}(x^*)$ by more than a positive constant multiple of $t^{\frac{\alpha}{2}} d(x,x^*)^{\alpha}$, at some point $\gamma^{x}_{t}$
along the geodesic from $x^*$ to $x$.
Note that $F_{\alpha}(x^*)=\inf_{x \in \cM}\,F_\alpha(x)$ and
$t^{\alpha/2}+(1-t)^{\alpha/2} \ge 1$ for all $t \in [0,1]$.
For the common choice $\eta=d^{2}$, i.e. $\alpha=2$, it follows from the (power transform) CN inequality that, for any $x \in \mathcal{M}\setminus\{x^*\}$,
\begin{equation*}
    b_{2}(x)=\sup_{t \in (0,1]} \frac{F_{2}(\gamma^{x}_{t})-F_{2}(x^*)}{t \cdot d(x,x^*)^{2}}
    \geq \sup_{t \in (0,1]} \frac{t^2 \cdot d(x,x^*)^{2}}{t \cdot d(x,x^*)^{2}}=1.
\end{equation*}
Thus, we may take $B_{2} = 1$ in this case and the proposition gives the usual variance inequality in \cref{Subsection 4.1}.
For $\eta=d^\alpha$ with $\alpha \in (1,2]$ in general, if $P \in \mathcal{P}_\alpha(\mathcal{M})$ satisfies $B_{\alpha}>0$, then (A1) and (A2) hold with $l(y,z)=\alpha2^{-\alpha+1}d(y,z)^{\alpha-1}, K=\alpha^{2} 2^{-2\alpha+2} B_{\alpha}^{-2+2/\alpha}$
and  $\beta=2-2/\alpha \in (0,1]$. Thus, in this general case as well, \cref{thm_fin_conv,thm_inf_conv} hold under the entropy conditions (B1) and (B2),
respectively. The theorems give that
\begin{align}\label{concent-mean-power-fin}
\mathbb{P}\left(d(x_n,x^*) \le 64\Big(\frac{K_{\alpha}^{\alpha/2}}{\alpha}\Big)^{1/(\alpha-1)}\cdot \Big(24\sqrt{AD}+\sqrt{\frac{2}{\Delta}} \Big)\cdot \frac{\sigma_X}{\sqrt{n}}\right) \ge 1-\Delta
\end{align}
for finite-dimensional NPC spaces $\cM$ and
\begin{align}\label{concent-mean-power-inf}
\mathbb{P}\left(d(x_n,x^*) \le 2\Big(\frac{K_{\alpha}^{\alpha/2}}{\alpha}\Big)^{1/(\alpha-1)}\cdot C_{A,\zeta} \cdot \rho_n \cdot \frac{\sigma_X}{\sqrt{\Delta}}\right) \ge 1-\Delta
\end{align}
for infinite-dimensional cases, where $K_{\alpha}=\alpha^2 2^{-2\alpha+2}B_{\alpha}^{-2+2/\alpha}$ and $\rho_n=n^{-1/2}$ if $0<\zeta<1$; $n^{-1/2}\cdot \log n$ if $\zeta=1$; $n^{-1/2\zeta}$ if $\zeta>1$.
Note that the concentration rates in terms of $\Delta$ and $n$ in \eqref{concent-mean-power-fin} and \eqref{concent-mean-power-inf} do not depend on $\alpha \in (1,2]$.

\begin{remark}\label{eta-choice}
There are other choices of $\eta$, not of the form $d^{\alpha}$, that may be of interest in statistics. For example, one may be interested in $\eta(x,y)=L_{\delta}(d(x,y))$, where $L_\delta$ with $\delta>0$ is the Huber loss defined by
\begin{equation*}
    L_{\delta}: [0,+\infty) \rightarrow [0,+\infty), \quad x \mapsto
    \begin{cases}
        x^{2}/2, &\quad 0 \le x \le \delta\\
        \delta \cdot (x-\delta/2), &\quad x > \delta.
    \end{cases}
\end{equation*}
This choice shares with $\eta=d^{\alpha}$ for $\alpha \in [1,2]$ the idea of combining the squared and absolute losses. Another example that may be of practical interest is the Kullback-Leibler divergence \cite{le2012asymptotic}, $\eta(\mu, \nu)=D_{KL} (\nu || \mu)$, when $(\cM,d)=(\mathcal{P}_{2}(\mathbb{R}),W_{2})$. The latter is an example of asymmetric functional. However, it seems difficult to prove the basic inequalities in (A1) and (A2) for general $\eta$. In particular, we are not aware of any type of bound for $\eta(\gamma_{t}, z)$ along a geodesic $\gamma:[0,1] \rightarrow \mathcal{M}$ for general $\eta$, which we need for the proof of (A2). To the best of our knowledge, even the results for $\eta=d^{\alpha}$ we present in \cref{thm_gen_CN,prop_gen_var_ineq}  are the first.
\end{remark}

\subsection{Metric entropy}\label{Subsection 3.2.3}

VC-type classes appear frequently in the study of empirical processes.
Our assumption (B1) on the complexity of $\cM$ in terms of the random entropy is crucial for the derivation of
non-asymptotic concentration properties of $x_n$.
It gives universal non-stochastic bounds to the random entropies $N(\tau, \mathcal{F}_\eta(\delta), \|\cdot\|_{2,P_{n}} )$.
The calculation of the (weak) VC index $D$ in (B1), i.e. the uniform control of the random covering numbers, is difficult in many cases (see Section 7.2 in \cite{van2014probability}).
A common technique to obtain $D$ is to exploit the combinatorial structure of the class of functions, provided that it is a VC subgraph class of functions, see \cite{boucheron2013concentration,gine2021mathematical,van2014probability} and references therein.
However, with a more explicit assumption (B1$'$) given below, which essentially characterizes the dimension of the underlying spaces,
we may calculate directly the (weak) VC index without combinatorial notions of complexity such as shattering.

\begin{itemize}\setlength{\itemindent}{1em}
\item[(B1$'$)] There are some constants $A_{1},D_1>0$ such that, for any $\tau \in (0,r]$,
\[
N(\tau,B(x^*,r),d) \leq \left(\frac{A_{1} r}{\tau} \right)^{D_1}.
\]
\end{itemize}

For a finite-dimensional normed space $\cM$, one may take $D_1={\rm dim}(\cM)$ irrespective of the underlying norm, since all norms in such a space are equivalent. On the contrary, $A_1$ depends on the choice of a metric $d$
and $A_1=3$ for the Euclidean norm when $\cM=\mathbb{R}^D$. In any case, (B1$'$) is for finite-dimensional $\cM$ and thus
the dependence of $A_1$ and $D_1$ on the metric $d$ does not need to be made explicit because the values of $A_1$ and $D_1$ do not affect the {\it convergence rates}
in \cref{thm_fin_conv} of \cref{Sec_Empirical means} and in \cref{thm_fin_mom,thm_fin_power_mom} of \cref{Sec_MoM} that are for finite-dimensional cases.

\begin{proposition}\label{prop_npc_cover}
Let $\eta=d^{\alpha}$ with $1<\alpha \le 2$. Assume (A2) and (B1$'$). Then (B1) holds with $A=A_1^{\alpha-1}$ and $D=D_{1}/(\alpha-1)${\rm:}
\[
N\left(\tau\|H_{\delta,\eta}\|_{2,P_{n}}, \mathcal{F}_\eta(\delta), \|\cdot\|_{2,P_{n}} \right) \leq \left(\frac{A_{1}}{\tau^{1/(\alpha-1)}}\right)^{D_{1}}, \quad 0<\tau \leq 1.
\]
In particular, when $\eta=d^{2}$ where (A2) is satisfied, (B1$'$) alone implies (B1) with $A=A_1$ and $D=D_{1}$.
\end{proposition}

Considering that the VC index $D_{\rm vc}$ introduced in \cref{assumptions} is usually larger than the dimension $D_1$ of the underlying space $\mathcal{M}$,
the second result in \cref{prop_npc_cover}
is striking as it states that the (weak) VC index $D$ equals $D_1$ in our framework when $\eta=d^2$. It is noteworthy that the right hand side of the inequality in \cref{prop_npc_cover}
does not involve any term related to $\delta$. This can be interpreted as that the growth of $\|H_{\delta,\eta}\|_{2,P_{n}}$ counterbalances
the increasing complexity of the class $\mathcal{F}_\eta(\delta)$ as $\delta$ gets larger.

When $\cM$ is a Riemannian manifold and $\eta=d^\alpha$ with $\alpha \in (1,2]$, the constant $A$ in (B1) is indispensably related to the {\it volume control problem}, which is one of the fundamental problems in geometry.
Indeed, the constant $A_1$ in (B$1'$) for a Riemannian manifold depends on how fast the volume of a ball grows as its radius increases, which relies on the sectional (or Ricci) curvature
of $\mathcal{M}$. The Bishop-G\"{u}nther inequality gives an upper bound to the volume change in terms of the sectional curvature, see Theorem 3.101 (ii) in \cite{gallot1990riemannian}.
For the reversed inequality, named as the Bishop-Gromov inequality, see \cite{villani2009optimal}.
Because of these inequalities, $A_1$ thus $A$ in (B1) becomes smaller as the curvature of $\mathcal{M}$ increases when $\eta=d^\alpha$ with $\alpha \in (1,2]$.

Contrary to the case of finite-dimensional $\cM$, a version of (B1$'$) is not true in many cases of infinite-dimensional $\cM$.
If $\cM$ is an infinite-dimensional normed space, then any closed ball is non-compact, so that there is some $\tau_{0}>0$ such that $\log N(\tau,B(x^*,r),d)=\infty$ for any $\tau<\tau_{0}$.
Therefore, the approach that mimics the finite-dimensional case does not work for infinite-dimensional $\cM$ in general.
However, for separable Hilbert spaces we may calculate directly the explicit constants in the assumption (B2),
$A=1/32$ and $\zeta=1$ as demonstrated in the following proposition.

\begin{proposition}\label{prop_hilb_cover}
Let $\cM$ be a Hilbert space and $\eta=d^2$ with $d(x,y)=\|x-y\|$. Then, for any probability measure $P \in \mathcal{P}_{2}(\mathcal{M})$,
\begin{equation*}
\log N\left(\tau\|H_{\delta,\eta}\|_{2,P}, \mathcal{F}_\eta(\delta), \|\cdot\|_{2,P}  \right) \leq \frac{1}{32\tau^{2}}, \quad 0<\tau \leq 1.
\end{equation*}
Furthermore, for the empirical measure $P_{n}$, it holds that
\begin{equation*}
\log N\left(\tau\|H_{\delta,\eta}\|_{2,P_{n}}, \mathcal{F}_\eta(\delta), \|\cdot\|_{2,P_{n}}  \right) \leq \frac{1}{32\tau^{2}}, \quad 0<\tau \leq 1.
\end{equation*}
\end{proposition}

\cref{prop_hilb_cover} may be used to verify (B2) with $\eta=d^2$ for Riemannian manifolds $(\cM,d)$.
Note that $d(x,y) \leq \|\log_{p}x-\log_{p}y\|$ for $\cM$ with non-negative curvature,
while $d(x,y) \geq \|\log_{p}x-\log_{p}y\|$ for $\cM$ with non-positive curvature, i.e. for Hadamard manifolds.
By embedding $\cM$ into the tangent space $T_{x^*}\mathcal{M}$ and applying \cref{prop_hilb_cover} to $T_{x^*}\mathcal{M}$,
one may argue that (B2) is satisfied with some $\zeta \leq 1$ for Riemannian manifolds with non-negative curvature,
and with some $\zeta \geq 1$ for Hadamard manifolds. In fact, $\zeta$ in (B2), termed as curvature complexity,
can be made smaller as the curvature of $\cM$ gets larger. The latter follows from the {\it Toponogov comparison theorem}:
the larger the sectional curvature of an underlying space $\mathcal{M}$ is, the slower the acceleration of the deviation between two geodesics emanating from a single point.

\subsection{Wasserstein space}
For a separable Banach space $(\mathcal{X},\|\cdot\|)$, $\mathcal{P}_2(\mathcal{X})$ is called {\it Wasserstein space}
and can be written as
\begin{equation*}
    \mathcal{P}_2(\mathcal{X})=\{\mu \in \mathcal{P}(\mathcal{X}): \int_{\mathcal{X}} \|x\|^2 {\rm d}\mu(x) < \infty \},
\end{equation*}
where $\mathcal{P}(\mathcal{X})$ denotes the set of all probability measures on $\cX$.
The Wasserstein space $\mathcal{P}_2(\mathcal{X})$ is equipped with the {\it Wasserstein distance}
\begin{equation*}
W_{2}(\mu, \nu)=\left(\inf_{\pi \in \Pi(\mu,\nu)} \int_{\mathcal{X} \times \mathcal{X}} \|x-y\|^{2} \rd \pi(x,y) \right)^{1/2}, \quad \mu, \nu \in \mathcal{P}_2(\mathcal{X})
\end{equation*}
where $\Pi(\mu,\nu)$ denotes the family of all probability measures on $\mathcal{M} \times \mathcal{M}$ with marginals $\mu$ and $\nu$.

The Wasserstein space $\mathcal{P}_2(\mathcal{X})$ for a general Banach space $\mathcal{X}$ has non-negative Alexandrov curvature at any probability measure
$\mu \in \mathcal{P}_2(\mathcal{X})$ that is absolutely continuous with respect to all non-degenerate Gaussian measures
\cite{ahidar2019rate, panaretos2020invitation}. For $\mathcal{X}=\mathbb{R}$, however, $\mathcal{P}_2(\mathbb{R})$ has
vanishing Alexandrov curvature \cite{kloeckner2010geometric}. Thus, the latter is an NPC space, and (A1) and (A2) are satisfied with $K=\beta=1$ and $l=W_2$
for the usual choice $\eta(\mu,\nu)= W_{2}(\mu, \nu)^{2}$, see \cref{Subsection 4.1}.
Even though $\mathcal{P}_2(\mathbb{R})$ is not compact, if we restrict ourselves to $\cM=\mathcal{P}_2([-B,B]) \subset \mathcal{P}_2(\mathbb{R})$ for $0<B<\infty$,
then $\cM$ is compact in $\mathcal{P}_2(\mathbb{R})$ (see Corollary 2.2.5 in \cite{panaretos2020invitation}) with a finite diameter:
$W_{2}(\mu, \nu)\le 2B$ for all $\mu, \nu \in \mathcal{P}_2([-B,B])$.
This implies that the Wasserstein ball $B(\mu^*,r) \subset \mathcal{P}_2(\mathbb{R})$ is way larger than $\mathcal{P}_2([-r/2,r/2])$, since the former set includes probability measures with non-compact support
and there is no hope that one can prove (B2) via a version of (B1$'$) when $\cM=\mathcal{P}_2(\mathbb{R})$.
Nonetheless, for $\cM=\mathcal{P}_2([-B,B])$ for some $B>0$, we may obtain a version of (B1$'$) for any $D_1>1$ due to Theorem A.1 in \cite{bolley2007quantitative}:
\begin{equation*}
N \left(\tau,\mathcal{P}_2([-B,B]),W_{2} \right)
\leq \left( \frac{\sqrt{16e}B}{\tau} \right)^{8B/\tau}.
\end{equation*}
This would give (B2) for some $A, \zeta>1/2$ that do not depend on $n$ as in the finite-dimensional case, see Subsection 2.2.4 of \cite{panaretos2020invitation} or Appendix A of \cite{bolley2007quantitative} for the explicit constants.

\section{Geometric-median-of-means}\label{Sec_MoM}

For empirical Fr\'echet means in non-compact metric spaces, polynomial concentration, as we derived in \cref{Sec_Empirical means}, is the best one can achieve.
In this section we introduce new estimators and we show that they have exponential concentration in general NPC spaces.
The definitions of the estimators are for general metric spaces $(\mathcal{M},d)$ and functionals $\eta$.

Let the random sample $\{X_1, \ldots, X_n\}$ be partitioned
into $k$ disjoint and independent blocks $\mathcal{B}_{1},\ldots$,$\mathcal{B}_{k}$ of size $m\ge n/k$. For each $1 \le j \le k$, define
\begin{align}\label{def-Fnj}
F_{n,j}(x)=\frac{1}{m}\sum_{X_i \in \mathcal{B}_j} \eta(x,X_{i}).
\end{align}
When $\cM$ is a Hilbert space, one may interpret $F_{n,j}(a) < F_{n,j}(b)$ for two points $a, b \in \cM$ as that $a$ is `closer'
than $b$ to the `center' of the $j$th block $\mathcal{B}_j$.
Indeed, in the case where $\cM=\mathbb{R}^D$ and $\eta(x,y)=|x-y|^2$,
\begin{align}\label{equiv-dist}
F_{n,j}(a) < F_{n,j}(b) \quad \mbox{if and only if} \quad |a-Z_j| < |b-Z_j|,
\end{align}
where $Z_j$ in general is the sample Fr\'echet mean of the block $\mathcal{B}_j$ defined by
\[
Z_j \in \argmin_{x \in \cM} F_{n,j}(x).
\]
More generally, when $\cM$ is a Hilbert space and $\eta(x,y)=\|x-y\|^2$, then $F_{n,j}(a) < F_{n,j}(b)$ is equivalent to $\|a-Z_{j}\| < \|b-Z_{j}\|$.
This follows from $F_{n,j}(x) =F_{n,j}(Z_j) + \|x-Z_j\|^2$.

\begin{definition}\label{defeatdef}
For $a,b \in \mathcal{M}$, we say that `$a$ defeats $b$' if $F_{n,j}(a) \le F_{n,j}(b)$ for more than $k/2$ blocks $\mathcal{B}_j$.
For $x \in \mathcal{M}$, let
\begin{align*}
S_{x}=\{a \in \mathcal{M} : \text{a defeats x} \}, \quad r_{x}=\argmin \{r>0 : S_{x} \subset B(x,r)\}.
\end{align*}
We call $S_{x}$ the `$x$-defeating region' and $r_x$ the `$x$-defeating radius'.
The new estimator $x_{MM}$ of $x^*$ is then defined by
\begin{equation}\label{mmdef}
x_{MM} \in \argmin_{x \in \mathcal{M}} r_{x}.
\end{equation}
We call it `geometric-median-of-means', or simply `median-of-means'
when there is no confusion.
\end{definition}

\begin{remark}\label{existence-of-xMM}
We note that `$a$ defeats $b$' if and only if ${\rm median}\{F_{n,j}(a) - F_{n,j}(b): 1 \le j \le k\} \le 0$, see \cite{lecue2020robust}.
The minimum in \eqref{mmdef} is always achieved, provided that $\eta:\mathcal{M} \times \mathcal{M} \rightarrow [0,+\infty)$ is continuous and for any $x \in \cM$, $\eta (x,y) \rightarrow \infty$ as $d(x,y) \rightarrow \infty$. For any $x \in \cM$, the $x$-defeating region $S_{x}$ is a closed and bounded subset of $\cM$ containing $x$, thus $r_{x} < +\infty$. This would entail that $x \mapsto r_{x}$ is a continuous function, and with the fact that $r_{x} \rightarrow \infty$ as $\min\{d(x,X_{1}),\dots, d(x,X_{n})\} \rightarrow \infty$, one may argue that the minimum of $r_x$ over $x \in \cM$ is attained at some point in $\cM$.
By definition, $x$ defeats itself so that $x \in S_x$ for all $x \in \cM$.
Also, `$a$ defeats $b$' does not always imply
`$b$ does not defeat $a$'. Both $a$ and $b$ can defeat each other, and if it happens then there exists at least one $j$ such that $F_{n,j}(a)=F_{n,j}(b)$.
Furthermore, $r_x\le r$ if and only if any point $a$ with $d(x,a)>r$ cannot defeat $x$ since
\[
r_{x}=\max\left\{d(x,a): a \in \mathcal{M} \text{ defeats } x\right\}.
\]
In the case where $\cM$ is a Euclidean space, the median-of-means may be interpreted in terms of Tukey depth, see \cite{hopkins2020mean}.
\end{remark}

In view of \eqref{equiv-dist}, our definition of `defeat' is a natural extension of the notion introduced in \cite{lugosi2019sub} for $\cM=\mathbb{R}^D$:
`$a$ defeats $b$' if
$|a-Z_{j}| \le |b-Z_{j}|$ for more than  $k/2$ blocks $\mathcal{B}_j$.
We note that, for curved metric spaces, the equivalence between $F_{n,j}(a) \le F_{n,j}(b)$ and $d(a,Z_{j}) \le d(b,Z_{j})$
is no longer valid in general. Our definition in terms of $F_{n,j}(x)$ is preferable to the one based on $d(x,Z_j)$ since the latter
needs the much more onerous computation of sample Fr\'echet means $Z_j$ for curved spaces.
Our definition dispenses with the calculation of $Z_j$ in all competitions between two points in $\cM$.

Although $d(a,Z_{j}) \le d(b,Z_{j})$ is not equivalent to $F_{n,j}(a) \le F_{n,j}(b)$ for curved spaces,
one may roughly interpret `$a$ defeats $b$' as that $a$ is closer than $b$ to the centers of more than half of the $k$ blocks,
for $\eta=d^\alpha$ with $\alpha \in (1,2]$.
The idea of minimizing the radius of defeating region is that, if $x$ is far away from $x^*$, and thus from the block centers $Z_j$,
then it is more likely that $x$ would be defeated by some point located far from $x$, i.e. $r_x$ would be large.
Since $x_{MM}$ is determined by the ordering relation based on $F_{n,j}$ rather than
by the magnitudes of $F_{n,j}$ themselves, it reflects the geometric structure of $\eta$
and inherits the characteristics of the Euclidean median of
$Z_{1}, \dots, Z_{k}$. Indeed, when $\cM=\mathbb{R}$ and $\eta(x,y)=|x-y|^2$, $x_{MM}$ in \cref{defeatdef} coincides with the usual sample median
of $Z_{1}, \dots, Z_{k}$.

\begin{figure}
\centering
\includegraphics[width=14cm,height=10cm]{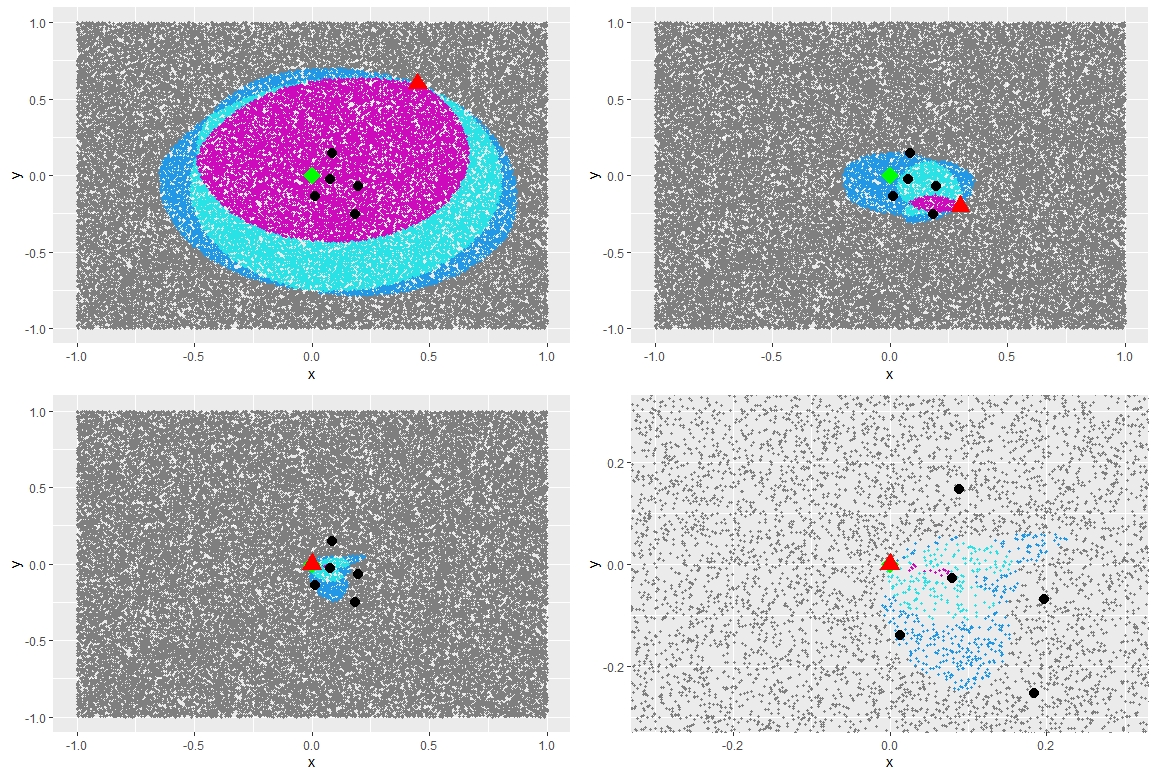}
\caption{Illustration of the $x$-defeating region $S_x$ for three choices of $x=\textcolor{red}{\blacktriangle}$: $\textcolor{red}{\blacktriangle}=(0.45,0.6)$ in the top-left,
$\textcolor{red}{\blacktriangle}=(0.3,-0.2)$ in the top-right and $\textcolor{red}{\blacktriangle}=(0,0)$ in the bottom-left panel.
For this a dataset $\{X_i: 1 \le i \le n\}$ of size $n=10,000$ was generated from a bivariate distribution on $\mathbb{R}^2$ with mean $\textcolor{green}{\blacklozenge}=(0,0)$, and it was partitioned into
$k=5$ blocks randomly. The block sample means $Z_{1}, \dots, Z_{5}$ are depicted as $\bullet$ points.
The bottom-right panel is the zoomed-in picture of the bottom-left.
In each panel with a given $\textcolor{red}{\blacktriangle}$, the color of each region indicates the `defeating score', against $\textcolor{red}{\blacktriangle}$, of the points in the region,
where the `defeating score' of a point $a$ against $\textcolor{red}{\blacktriangle}$
equals the number of blocks $\mathcal{B}_{j}$ such that $F_{n,j}(a) \le F_{n,j}(\textcolor{red}{\blacktriangle})$. The violet region is for the score 5, the sky blue for 4, the blue for 3 and
the gray for the scores $\leq 2$.
Thus, the union of violet/sky-blue/blue colored regions is the $x$-defeating region $S_{\textcolor{red}{\blacktriangle}}$ in each panel.
}
\label{defeatfigure}
\end{figure}

To illustrate how $x_{MM}$ works, we simulated $n=10,000$ data points from a bivariate distribution
and chose $k=5$ for the number of blocks.
In Figure~\ref{defeatfigure} we depicted them on $[-1,1]^2$ and also $Z_j$ ($\bullet$) for $1 \le j \le 5$.
The figure demonstrates that $r_{x}$, which is the radius of the smallest ball centered at $x=\textcolor{red}{\blacktriangle}$ covering the `violet/sky-blue/blue' regions,
tends to decrease as $x \in \mathcal{M}$ gets closer to the Fr\'echet mean $x^*=\textcolor{green}{\blacklozenge}$.
To see how sensitive $x_{MM}$ is to the change of data points, imagine that the data points
in a single block changes completely to arbitrary values. This would change only one $F_{n,j}(\cdot)$
among the five, regardless how extreme the change of the data points is.
Since the points $a$ in the violet and sky-blue regions, respectively, have $F_{n,j}(a) \le F_{n,j}(\textcolor{red}{\blacktriangle})$ for $5$ and $4$ blocks with the original dataset, they still defeat $x=\textcolor{red}{\blacktriangle}$ with the modified dataset.
From this one may infer that there would be no significant change in the ordering of $r_x$ across $x \in \cM$.
This consideration suggests that $x_{MM}$ is more robust than $x_n$ to large deviation of a few blocks, which results in $x_{MM}$ having stronger concentration than $x_{n}$,
provided that the number of blocks ($k$) is sufficiently large.
The latter has been evidenced for $\cM=\mathbb{R}$ by \cite{catoni2012challenging,devroye2016sub} and for $\cM=\mathbb{R}^D$ by \cite{lugosi2019sub}.

In the next two subsections, we make precise the above heuristic discussion for NPC spaces with $\eta=d^\alpha$ for $\alpha \in (1,2]$.

\subsection{Common choice \texorpdfstring{$\eta=d^{2}$}{: squared distance}}\label{Section 5.1}

Let $X_{1},\dots, X_{n}$ be i.i.d. random elements taking values in an NPC space $(\mathcal{M},d)$ with finite second moment. Here, we focus on the case $\eta=d^{2}$. The following theorem is essential for deriving an exponential concentration for $x_{MM}$ when $\cM$ is of finite dimension.

\begin{theorem}\label{thm_fin_mom}
Assume (B1) with some constants $A,D>0$. Let $\Delta \in (0,1)$ and $q \in (0,1/2)$.
Let $k$ denote the number of blocks $\mathcal{B}_j$.
If $k=\lceil 1/(2q^2) \log(1/\Delta)\rceil$, then it holds that, with probability at least $1-\Delta$,
$x^*$ defeats all $x \in \mathcal{M}$ with $d(x,x^*)>R_q$ but any such $x$ does not defeat $x^*$, where
\begin{align}\label{mmbound}
R_q=C_q \sigma_{X} \sqrt{\frac{\log(1/\Delta)}{n}},\quad
C_q= \frac{32\sqrt{2}}{q} \left(24\sqrt{AD}+\frac{2}{\sqrt{1-2q}} \right).
\end{align}
\end{theorem}

Let $\cE$ denote an event where, for all $x$ with $d(x,x^*)>R_q$, $x^*$ defeats $x$ but $x$ does not defeat $x^*$. On $\cE \cap\{d(x_{MM},x^*)>R_q\}$,
one has $x^* \in S_{x_{MM}}$, which implies $S_{x_{MM}} \nsubseteq B(x_{MM},R_q)$ so that
$r_{x_{MM}} >R_q$. On $\cE$, one also gets that $x \notin S_{x^*}$ for all $x$ with $d(x,x^*)>R_q$, which implies $S_{x^*} \subset B(x^*,R_q)$ so that
$r_{x^*} \le R_q$ on $\cE$. By the definition of $x_{MM}$, it holds that $r_{x_{MM}} \le r_{x^*}$, however. This means that
\[
\mathbb{P}\big(\cE \cap\{d(x_{MM},x^*)>R_q\}\big)=0.
\]
The foregoing arguments give the following corollary of \cref{thm_fin_mom}.

\begin{corollary}\label{cor_fin_mom}
Assume (B1) with some constants $A,D>0$. Let $\Delta \in (0,1)$ and $q\in(0,1/2)$. Let $k$ denote the number of blocks $\mathcal{B}_j$.
If $k=\lceil 1/(2q^2) \log(1/\Delta)\rceil$, then it holds that $d(x_{MM},x^*) \leq R_q$ with probability at least $1-\Delta$,
where $R_q$ is the constant defined at \eqref{mmbound}.
\end{corollary}

\begin{remark}
Note that the condition $\Delta \in [e^{-2q^{2}n},1)$ is latent in \cref{thm_fin_mom} and also in \cref{thm_fin_power_mom,thm_inf_mom,thm_inf_power_mom,cor_fin_mom,cor_fin_power_mom,cor_inf_mom,cor_inf_power_mom}, since the number of blocks $k=\lceil 1/(2q^2) \log(1/\Delta)\rceil \le n$. When $\cM=\mathbb{R}$, it is known that one should impose $\Delta \in [\Delta_{min},1)$ for some $\Delta_{min}>0$ to achieve a sub-Gaussian performance, see \cite{devroye2016sub}.
\end{remark}

The constant factor $C_q$ in the radius of concentration $R_q$ depends on $q \in (0,1/2)$.
Taking too small (large) $q$ close to $0$ ($1/2$) leads to too large (small) number of blocks $k$, which results in inflating the constant $C_q$
and impairing the concentration property of $x_{MM}$. There is an optimal $q$ in the interval $(0,1/2)$ that minimizes $C_q$ since $C_q$ is a smooth function of $q \in (0,1/2)$
and diverges to $+\infty$ as $q$ approaches either to $0$ or to $1/2$.
We note that $x_{MM}$ with too small $k$ is not much differentiated from the empirical Fr\'echet mean $x_n$, while with too large $k$
the block Fr\'echet means $Z_j$ would be scattered and thus there would be no guarantee that points $x$ close to $x^*$ have small $x$-defeating radius $r_x$.

The following theorem is for infinite-dimensional $\cM$ and also gives an exponential concentration for $x_{MM}$.
\begin{theorem}\label{thm_inf_mom}
Assume (B2) with some constants $A>0$ and $\zeta \ge 1$. Let $\Delta \in (0,1)$ and $q\in(0,1/2)$. Let $k$ denote the number of blocks $\mathcal{B}_j$.
If $k=\lceil 1/(2q^2) \log(1/\Delta)\rceil$, then it holds that, with probability at least $1-\Delta$,
$x^*$ defeats all $x \in \mathcal{M}$ with $d(x,x^*)>R_q$ but any such $x$ does not defeat $x^*$,
where
\begin{equation}\label{mmbound2}
\begin{split}
R_{q,\zeta}=
\begin{cases}
\displaystyle c_{q, 1} \cdot \sigma_X \cdot \log n\cdot \sqrt{\frac{\log (1/\Delta)}{n}} \quad &\text{if }\zeta = 1\\ 
\quad \\
\displaystyle c_{q, \zeta} \cdot \sigma_X \cdot \left(\frac{\log (1/\Delta)}{n}\right)^{1/2\zeta} \quad &\text{if }\zeta>1\\ 
\end{cases} \\
\end{split}
\end{equation}
where $c_{q, \zeta}=\frac{2 C_{A, \zeta}}{q\sqrt{1-2q}}$ with $C_{A, \zeta}$ appearing in \cref{thm_inf_conv}.
\end{theorem}

\begin{corollary}\label{cor_inf_mom}
Assume (B2) with some constants $A>0$ and $\zeta \ge 1$. Let $\Delta \in (0,1)$ and $q\in(0,1/2)$. Let $k$ denote the number of blocks $\mathcal{B}_j$.
If $k=\lceil 1/(2q^2) \log(1/\Delta)\rceil$, then it holds that $d(x_{MM},x^*) \leq R_{q,\zeta}$ with probability at least $1-\Delta$,
where $R_{q,\zeta}$ is the constant defined at \eqref{mmbound2}.
\end{corollary}

As in the case of the empirical Fr\'echet mean $x_n$ for infinite-dimensional $\cM$, see \eqref{concent-mean-power-inf},
decreasing the curvature of $\cM$ (increasing $\zeta$) results in slowing down the rate of convergence of $x_{MM}$ to $x^*$.
We can also make a similar remark for the dependence of the constant factor $c_{q,\zeta}$ on $q \in (0,1/2)$ as in the discussion of \cref{cor_fin_mom}.
In the infinite-dimensional case, however, $c_{q,\zeta}$ is minimized at some point $q \in (0,1/2)$.

We note that the constants $C_q$ and $c_{q,\zeta}$ in \cref{thm_fin_mom,thm_inf_mom},
respectively, may not be optimal. One might improve them
by carefully sharpening of various inequalities in the proofs of the theorems.
Rather than optimizing the constants, we focus on deriving {\it exponential} concentration.
It is also noteworthy that our results do not involve terms such as
$\operatorname{tr}(\Sigma_{X})$, as opposed to the radius of concentration derived by Lugosi \cite{lugosi2019sub} for the case $\cM=\mathbb{R}^D$,
since we do not assume any differential structure for the underlying NPC space.
The rates of concentration in \cref{cor_fin_mom,cor_inf_mom} are not optimal when $\cM$ is a Hilbert space unless $\cM=\mathbb{R}$.
In the latter case, the optimal rate of concentration is known to be $O ( \sqrt{\text{tr}(\Sigma_{X})/n}+\sqrt{\|\Sigma_{X}\| \log(1/\Delta)/n} )$ as in \eqref{sub-Gauss_conc}.
It is noteworthy that $\sigma_{X}^2=\text{tr}(\Sigma_{X})$ when $\cM$ is a Hilbert space. However, metric spaces without a differential structure do not have
an equivalent of the covariance matrix $\Sigma_{X}$ in general. Moreover, $\text{tr}(\Sigma_{X})$ in \cite{lugosi2019sub} arises from the \textit{dual Sudakov inequality},
which accounts for the covering number of a sphere $r \cdot S^{D-1}$ with respect to the norm $\|\cdot\|_{2,P}$ in terms of $r$ and $\text{tr}(\Sigma_{X})$.
The inequality is based on the linear structure of $\mathbb{R}^{D}$ and the fact that  $\|\cdot\|_{2,P}$ is translation invariant, therefore
it is no longer valid for non-vector spaces. Hence, even for Hadamard manifolds where a differential structure is available, it seems intractable to obtain an inequality
that corresponds to the dual Sudakov inequality.

Now, we present a theorem that gives the {\it breakdown point} of $x_{MM}$. The breakdown point of an estimator is the smallest proportion of data corruption that can upset the estimator completely.
It tells the level of resistance by an estimator against data corruption and is a popular measure of robustness in statistics. Let $\cX_n=\{X_1,\ldots,X_n\}$.
For a configuration $\{i(1), \ldots, i(\ell)\}\subset \{1, 2, \ldots, n\}$, let
$\tilde \cX_n(i(1), \dots, i(\ell))$ denote the modification of $\cX_n$
for which $X_{i(j)}$ for $1 \le j \le \ell$ in $\cX_n$ are replaced by $\tilde X_{i(j)}$,
respectively. For an estimator $\hat{x}$, the breakdown point of $\hat x$ is defined as
\begin{align*}
\varepsilon_{n}^*&:=\frac{1}{n} \min \Big\{\ell: \mbox{there exists a dataset } \cX_n \mbox{ and a configuration } \{i(1), \ldots, i(\ell)\} \mbox{ such that }\\
&\hspace{4.5cm}\sup_{\tilde X_{i(1)}, \dots, \tilde X_{i(\ell)}} d\left(\hat{x}(\cX_n), \,\hat{x}\big(\tilde \cX_n(i(1), \dots, i(\ell))\big)\right)
=\infty \Big\}.
\end{align*}
For the above definition to make sense, we consider the case where ${\rm diam}(\cM) = \infty$. The following theorem demonstrates that the breakdown point $\varepsilon_{n}^*$
of $x_{MM}$ for an NPC space $(\cM,d)$ equals that of the median-of-means tournament for $\cM=\mathbb{R}^{D}$.

\begin{theorem}\label{mom_breakdown}
Let $(\cM,d)$ be an NPC space where $X_{1},\dots, X_{n}$ take values. Let $k$ denote the number of blocks $\mathcal{B}_j$.
Then, the breakdown point of $x_{MM}$ associated with $\eta=d^2$
is independent of partition $\{\cB_j: 1 \le j \le k\}$ and equals
$\varepsilon_{n}^*=n^{-1}\cdot \lceil (k+1)/2\rceil$.
\end{theorem}

One may be interested in studying the concentration properties of geometric-median-of-means when some portion of the dataset are corrupted. This has been done by
\cite{depersin2021robustness} for $\cM=\mathbb{R}^D$. Its extension to NPC spaces is a challenging topic for future study.

\subsection{Cases with \texorpdfstring{$\eta=d^\alpha$}{the power transform metric}}\label{Section 5.2}

Here, we consider a more general setting where $\eta=d^{\alpha}$ for $1<\alpha \leq 2$.
We note that the CN inequality in \cref{Subsection 4.1} plays an important role in establishing \cref{thm_fin_mom,thm_inf_mom}. For the general case with $\eta=d^{\alpha}$,
we use the power transform CN inequality established in \cref{thm_gen_CN}.

The general estimators are built on the following notion of `defeat by fraction'.
The definition applies not only to $\eta=d^\alpha$ but also to a general measurable function $\eta:\mathcal{M}\times\mathcal{M} \rightarrow \mathbb{R}$.

\begin{definition}\label{weakdefeatdef}
Let $\rho$ be a positive real number. For $a,b \in \mathcal{M}$, we say that `$a$ defeats $b$
by fraction $\rho$'
if $F_{n,j}(a) \le \rho \cdot F_{n,j}(b)$ for more than $k/2$ blocks $\mathcal{B}_j$.
For $x \in \mathcal{M}$, let
\begin{align*}
S_{\rho,x}&=\{a \in \mathcal{M} : \text{a defeats x by fraction } \rho\},\\
\quad r_{\rho,x}&=\min \{r>0 : S_{\rho,x} \subset B(x,r)\}\\
&= \max\{d(x,a): a \in \mathcal{M} \text{ defeats } x \mbox{ by fraction } \rho\}.
\end{align*}
We call $S_{\rho,x}$ the `$x$-defeating-by-$\rho$ region' and $r_x$ the `$x$-defeating-by-$\rho$ radius'.
The estimator $x_{\rho,MM}$ of $x^*$ is then defined by
\begin{equation*}
x_{\rho,MM} \in \argmin_{x \in \mathcal{M}} r_{\rho,x}.
\end{equation*}
We call it `$\rho$-geometric-median-of-means', or simply `$\rho$-median-of-means' if there is no confusion.
\end{definition}

Clearly, the case $\rho=1$ in the above definition coincides with \cref{defeatdef}.
By defintion, for any $0<\rho_{1}<\rho_{2}$, if $a$ defeats $b$ by fraction $\rho_1$, then $a$ defeats $b$ by fraction $\rho_2$.
Therefore, for any fixed $x \in \mathcal{M}$, the $x$-defeating-by-$\rho$ region $S_{\rho,x}$ increases as $\rho$ increases, and $\rho \mapsto r_{\rho,x}$ is a monotone increasing function.

For $0<\rho<1$, the $x$-defeating-by-$\rho$ region does not contain $x$ since $S_{\rho,x}$ collects those points in $\cM$ that are `strictly better' than $x$.
If $\rho$ is too small, $S_{\rho,x}$ can be an empty set for some $x \in \mathcal{M}$, in which case $r_{\rho, x}=0$.
We note that the two events `$a$ defeats $b$ by fraction $\rho$' and `$b$ defeats $a$ by fraction $1/\rho$' do not complement each other,
but either of the two always occurs. Both can occur simultaneously, and if so then there exists at least one $j$ such that $F_{n,j}(a) = \rho \cdot F_{n,j}(b)$.
As in the case of $\rho=1$, the minimum of $r_{\rho,x}$ over $x \in \cM$ is attained at some point in $\cM$ when $\eta:\mathcal{M} \times \mathcal{M} \rightarrow \mathbb{R}$ is continuous.

To state a generalization of \cref{thm_fin_mom} to the case $\eta=d^\alpha$, put
\begin{equation*}
M_{\alpha,\rho}=\sup\left\{ \delta^{1-\alpha/2} t^{\alpha/2} (1-t)^{\alpha/2} : 0 < t < 1, \delta > 0, \frac{1-(1+\delta)^{1-\alpha/2} (1-t)^{\alpha/2}}{(1+\delta)^{1-\alpha/2}\, t^{\alpha/2}} \geq \rho\right\}.
\end{equation*}
Note that $M_{\alpha,\rho}=1/4$ for $\alpha=2$ and $\rho \le 1$ since for any $0 < t < 1$ and $\delta > 0$,
\[
\frac{1-(1+\delta)^{1-2/2} (1-t)^{2/2}}{(1+\delta)^{1-2/2} \,t^{2/2}} =\frac{t}{t}=1.
\]
However, for $0<\alpha<2$, we note that $t^{\alpha/2}+(1-t)^{\alpha/2}>1$ for all $0<t<1$ and thus
\begin{equation}\label{Malphac}
\frac{1-(1+\delta)^{1-\alpha/2} (1-t)^{\alpha/2}}{(1+\delta)^{1-\alpha/2} \,t^{\alpha/2}} < 1
\end{equation}
for all $0 <t < 1$ and $\delta >0$.
Hence, taking $\rho\ge 1$ when $\eta=d^\alpha$ for $0<\alpha<2$,
as (\ref{Malphac}) shows, would give $M_{\alpha,\rho}=\sup \emptyset=-\infty$.
In fact, we find that the derivation of exponential concentration is intractable for $x_{\rho,MM}$ with $\rho \ge 1$ when $1<\alpha<2$, which is why
we introduce the new notions of `defeat by fraction' and `$\rho$-geometric-median-of-means estimator'.
\cref{Malphafigure} demonstrates the shapes of $M_{\alpha, \rho}$ as a function of $\rho$ for several choices of $\alpha$.
It also depicts $M_{\alpha, \rho}^{-1/\alpha}$ on the log scale that appears in the constant factors in the concentration inequalities
in the following theorems and corollaries.

\begin{figure}
\centering
\includegraphics[width=10cm,height=5cm]{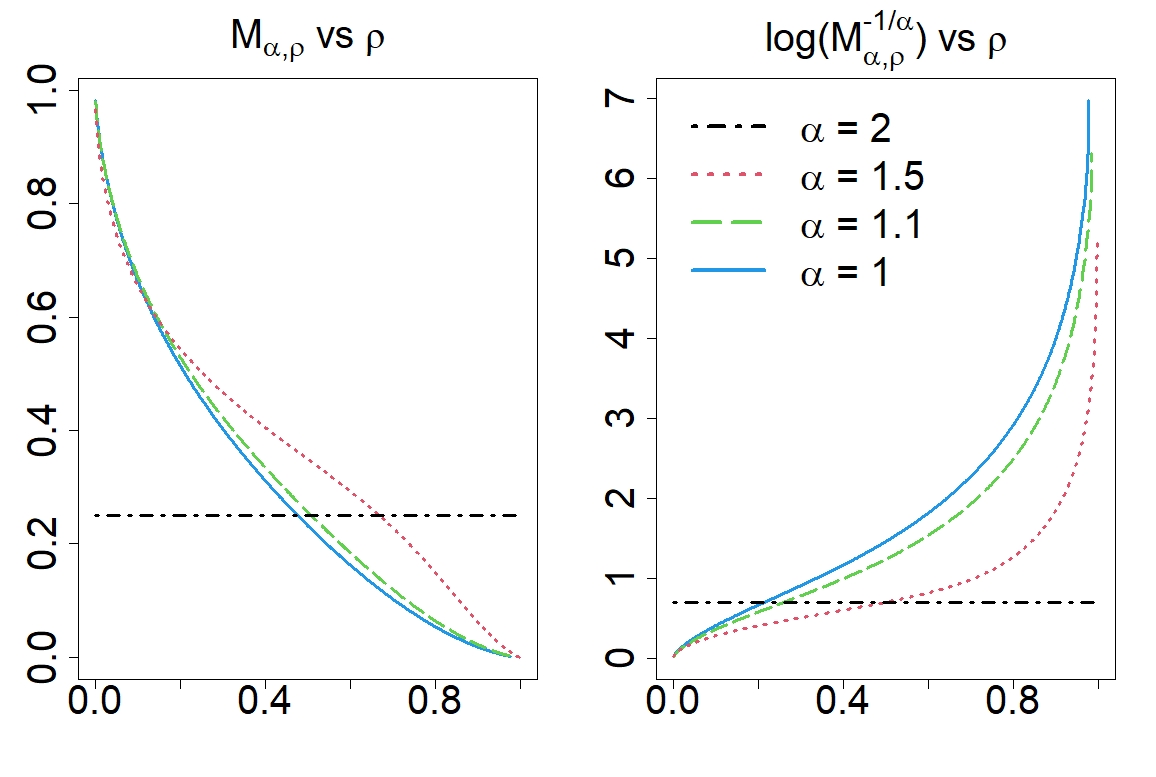}
\caption{The shapes of $M_{\alpha, \rho}$ (left) and $\log M_{\alpha, \rho}^{-1/\alpha}$ (right)
as functions of $\rho$ for $\alpha=1/1.1/1.5/2$ (solid/dashed/dotted/dot-dashed).}
\label{Malphafigure}
\end{figure}

\begin{theorem}\label{thm_fin_power_mom}
Assume (B1) with some constants $A,D>0$ and that there exists a constant $B_{\alpha}>0$ such that
\begin{align}\label{cond-var-ineq}
d\left(x, x^{*}\right)^{\alpha} \leq \frac{1}{B_{\alpha}} \int_{M}\left(d(x, y)^{\alpha}-d\left(x^{*}, y\right)^{\alpha} \right)\mathrm{d} P(y).
\end{align}
Let $\rho \in (0,1)$ for $\alpha \in (1,2)$ or $\rho=1$ when $\alpha=2$.
Also, let $\Delta \in (0,1)$ and $q\in(0,1/2)$.
Put $K_{\alpha}=\alpha^2 2^{-2\alpha+2}B_{\alpha}^{-2+2/\alpha}$. Let $k$ denote the number of blocks $\mathcal{B}_j$.
If $k=\lceil 1/(2q^2) \log(1/\Delta)\rceil$, then it holds that, with probability at least $1-\Delta$,
$x^*$ defeats by fraction $1/\rho$ all $x \in \mathcal{M}$ with $d(x,x^*)>R_{q,\alpha,\rho}$ but any such $x$ does not defeat $x^*$ by fraction $\rho$,
where
\begin{equation}\label{mmbound3}
\begin{split}
&R_{q,\alpha,\rho}=C_{q,\alpha,\rho} \,\sigma_{X}\sqrt{\frac{\log(1/\Delta)}{n}},\\
&C_{q,\alpha,\rho}= M_{\alpha,\rho}^{-1/\alpha}\cdot \frac{16\sqrt{2K_{\alpha}}}{q} \left(24\sqrt{AD}+\frac{2}{\sqrt{1-2q}}\right).
\end{split}
\end{equation}
\end{theorem}

Recall that \cref{prop_gen_var_ineq} gives a sufficient condition for the existence of $B_{\alpha}>0$ such that (\ref{cond-var-ineq}) holds.
Also, we note that (\ref{cond-var-ineq}) holds with $B_{\alpha}=1$ when $\alpha=2$, see \cref{Subsection 4.1}.
Thus, when $\alpha=2$ and $M_{\alpha,\rho}=1/4$, we have $K_{\alpha}=1$ so that \cref{thm_fin_power_mom} with $\rho=1$ reduces to \cref{thm_fin_mom}.
The following corollary may be derived from \cref{thm_fin_power_mom} as \cref{cor_fin_mom} is from \cref{thm_fin_mom}.

\begin{corollary}\label{cor_fin_power_mom}
Assume the conditions and consider the ranges of $(\rho, \alpha)$, $\Delta$ and $q$ in \cref{thm_fin_power_mom}.
Let $k$ denote the number of blocks $\mathcal{B}_j$.
If $k=\lceil 1/(2q^2) \log(1/\Delta)\rceil$, then it holds that $d(x_{\rho,MM},x^*) \leq R_{q,\alpha,\rho}$ with probability at least $1-\Delta$,
where $R_{q,\alpha,\rho}$ is the constant defined at \eqref{mmbound3}.
\end{corollary}

The constant factor $C_{q,\alpha,\rho}$ depends on $q$ and $\rho$. As in \cref{cor_fin_mom}
for $x_{MM}$, it is minimized at some point $q \in (0,1/2)$. The minimizing $q$ depends on
$A$ and $D$, but is independent of $\alpha$ and $\rho$. As for the dependence on $\rho$,
we note that $\rho \in (0,1)\mapsto C_{q,\alpha,\rho} \in (0,+\infty)$ is an increasing function
when $1<\alpha<2$,
as is well illustrated by the right panel of \cref{Malphafigure}.
The increasing speed gets extremely fast as $\rho$ approaches to $1$.
Since taking a smaller $\rho$ shrinks the defeating regions $S_{\rho,x}$, it results in
having $x_{\rho,MM}$ stay closer to $x^*$, which explains the result that the radius of
concentration $R_{q,\alpha,\rho}$ gets smaller for smaller $\rho$.

Below, we present versions of \cref{thm_fin_power_mom} and \cref{cor_fin_power_mom} when $\cM$ is of infinite-dimension satisfying the entropy condition (B2).
Again, when $\alpha=2$, we have $K_{\alpha}=1$ and $M_{\alpha,\rho}=1/4$ so that \cref{thm_inf_power_mom} with $\rho=1$ reduces to \cref{thm_inf_mom}.

\begin{theorem}\label{thm_inf_power_mom}
Assume (B2) with some constants $A>0$ and $\zeta\ge 1$ and that there exists a constant $B_{\alpha}>0$ such that \eqref{cond-var-ineq} holds.
Consider the ranges of $(\rho, \alpha)$, $\Delta$ and $q$ in \cref{thm_fin_power_mom}.
Put $K_{\alpha}=\alpha^2 2^{-2\alpha+2}B_{\alpha}^{-2+2/\alpha}$. Let $k$ denote the number of blocks $\mathcal{B}_j$.
If $k=\lceil 1/(2q^2) \log(1/\Delta)\rceil$, then it holds that, with probability at least $1-\Delta$,
$x^*$ defeats by fraction $1/\rho$ all $x \in \mathcal{M}$ with $d(x,x^*)>R_{q,\alpha,\rho}$ but any such $x$ does not defeat $x^*$ by fraction $\rho$,
where
\begin{equation}\label{mmbound4}
\begin{split}
R_{q,\alpha,\rho,\zeta}&=
\begin{cases}
\displaystyle c_{q,\alpha,\rho, 1} \cdot \frac{\log n}{n^{1/2}} \cdot \sigma_X \cdot \sqrt{\log \frac{1}{\Delta}}  \quad &\text{if}\;\;\zeta = 1\\ 
\quad \\
\displaystyle c_{q,\alpha,\rho, \zeta} \cdot \frac{1}{n^{1/2\zeta}}\cdot \sigma_X \cdot \left(\log \frac{1}{\Delta}\right)^{1/2\zeta} \quad &\text{if}\;\;\zeta>1,
\end{cases}\\[2mm]
c_{q,\alpha,\rho, \zeta}&= K_{\alpha}^{1/2} M_{\alpha,\rho}^{-1/\alpha}\cdot\frac{C_{A,\zeta}}{q \sqrt{1-2q}}
\end{split}
\end{equation}
and $C_{A, \zeta}$ is the constant that appears in \cref{thm_inf_conv}.
\end{theorem}

\begin{corollary}\label{cor_inf_power_mom}
Assume the conditions and consider the ranges of $(\rho, \alpha)$, $\Delta$ and $q$ in \cref{thm_inf_power_mom}.
Let $k$ denote the number of blocks $\mathcal{B}_j$.
If $k=\lceil 1/(2q^2) \log(1/\Delta)\rceil$, then it holds that $d(x_{\rho,MM},x^*) \leq R_{q,\alpha,\rho,\zeta}$ with probability at least $1-\Delta$,
where $R_{q,\alpha,\rho,\zeta}$ is the constant defined at \eqref{mmbound4}.
\end{corollary}

From \eqref{concent-mean-power-fin} and \eqref{concent-mean-power-inf} in \cref{Subsection 4.2}
we have observed that the concentration rates for the empirical Fr\'echet mean $x_n$ in terms of $\Delta$ and $n$ do not depend on $\alpha \in (1,2]$.
This is also the case with the geometric-median-of-means estimators $x_{MM}$ and $x_{\rho,MM}$, which can be seen by comparing \cref{cor_fin_mom,cor_inf_mom} with \cref{cor_fin_power_mom,cor_inf_power_mom}, respectively.
The dependence pattern of the rate of convergence of $x_{\rho,MM}$ on
the curvature complexity $\zeta$ is the same as $x_n$ and $x_{MM}$.
Also, the dependence of $c_{q,\alpha,\rho,\zeta}$ on $\rho$
is the same as in the finite-dimensional case. For the dependence on $q$,
as in the case of $x_{MM}$, the constant factor is minimized at some point $q \in (0,1/2)$.

\begin{remark}
For NPC spaces $\cM$ with $\eta=d^2$, the curvature complexity $\zeta$ is greater than or equal to $1$ (\cref{prop_hilb_cover}).
However, $\zeta$ may be strictly less than $1$ when $\eta=d^{\alpha}$ with $1<\alpha<2$. In the latter case, one may prove that the radius of concentration $R_{q,\alpha,\rho,\zeta}$
in \cref{thm_inf_power_mom} is given by
\[
R_{q,\alpha,\rho,\zeta}=\displaystyle c_{q,\alpha,\rho, \zeta} \cdot \frac{1}{n^{1/2}} \cdot \sigma_X \cdot \sqrt{\log \frac{1}{\Delta}}, \quad 0<\zeta < 1
\]
for the same constant $c_{q,\alpha,\rho, \zeta}$ given at \eqref{mmbound4}.
\end{remark}

\section{Concluding remarks}

Our results can be applied to any NPC spaces of finite or infinite dimension, such as Hilbert spaces, hyperbolic spaces, manifolds of SPD matrices,
and the Wasserstein space $\mathcal{P}_{2}(\mathbb{R})$, etc. Our work is an extensive generalization of previous works on the median-of-means method.
It is the first attempt that extends the notion of median-of-means to a general class of metric spaces with a rich class of metrics, and
derives exponential concentration for the extended notions of median-of-means in such a general setting.
As we discussed in this paper, we stress that the sample Fr\'echet mean has poor concentration
for non-compact or negatively curved spaces. For such spaces, our geometric-median-of-means estimators are
efficient antidotes to the sample Fr\'echet mean.

For Euclidean or Hilbertian spaces $\cM$, there is a large body of works that study sub-Gaussian mean estimators under only a second moment condition, see \cite{lugosi2019sub,lugosi2019mean,hopkins2020mean} and references therein. For general metric spaces, however, the definition of sub-Gaussianity itself is not available. It is a challenging future topic to generalize the notion of sub-Gaussian performance to more general metric spaces
and investigate the concentration properties of the corresponding empirical Fr\'echet means with the extended notion of sub-Gaussianity.

We admit that there is an issue of algorithmic feasibility with the geometric-median-of-means estimator studied in our paper.
The computational issue is also present in the Euclidean case for the median-of-means tournament estimator, see \cite{lugosi2019sub}.
There are some alternative proposals that are equipped with an efficient algorithm. These include
the geometric median of \cite{minsker2015geometric}, Catoni-Giulini estimator of \cite{catoni2018dimension}, the Hopkins' estimator \cite{hopkins2020mean} and those in the follow-up studies by \cite{cherapanamjeri2019fast, depersin2022robust, lei2020fast}. However, all these estimators are proposed and studied in the case where $\cM$ is a Euclidean or a Banach space.
In particular, the estimators studied in \cite{hopkins2020mean,cherapanamjeri2019fast, depersin2022robust, lei2020fast}
combine the idea of semi-definite programming (SDP) and \textit{r-centrality} (see \cite{hopkins2020mean} for definition), which requires an inner product structure for the underlying space.
For some spaces that admit a tangential structure equipped with a bi-invariant metric, one may borrow the idea of Hopkins \cite{hopkins2020mean} to find a robust Fr\'echet mean estimator equipped with an efficient algorithm.
For instance, if $\cM=\mathcal{S}_{D}^{+}$, the space of symmetric positive-definite matrices, and it is endowed with the log-Euclidean metric, one might project the dataset onto the tangent at the identity $I_{D}$ via the logarithmic map, compute the Hopkins' estimator from the projected data, and then transform the result back to $\mathcal{S}_{D}^{+}$ via the exponential map.
It is straightforward to show that the resulting estimator of the Fr\'echet mean is consistent. However, this is not an estimator of our interest in this paper, and the special treatment would restrict the study to Riemannian manifolds. It is a challenging topic of study to develop an efficient algorithm for the geometric-median-of-means estimator in the general setting of NPC spaces.

\begin{appendix}
\section{Proofs of theorems}\label{appn}
In the Appendix, we give the complete proofs of Theorems~\ref{thm_fin_conv}--\ref{thm_inf_power_mom}. The proofs of lemmas and Propositions~\ref{thm_gen_CN}--\ref{prop_hilb_cover} in \cref{Section 4}, some of which are used in the proofs of Theorems, can be found in the Supplementary Material.
Throughout the Appendix and the Supplementary Material, we often denote $\int_{\mathcal{S}} f(y) \,\mathrm{d} Q(y)$ simply by $Qf$
for a measurable space $(\mathcal{S}, \mathcal{B})$, a probability measure $Q$ on $\mathcal{B}$
and a measurable function $f: \mathcal{S} \rightarrow \mathbb{R}$.
For instance, $Pf=\mathbb{E}\,(f(X))$ and $P_n f=n^{-1}\sum_{i=1}^n f(X_i)$.
We also suppress the dependence on $\eta$ of $\cM_\eta(\delta)$ and other associated terms.
All lemmas being referred to by `S.xx' are those in the Supplementary Material.

\subsection{Proofs of theorems in \cref{Sec_Empirical means}}\label{proofs-Sec3}

\begin{proof}[Proof of \cref{thm_fin_conv}]
Define $\delta_{n}= P(\eta(x_{n}, \cdot)-\eta\left(x^{*}, \cdot\right))$ and
\begin{align*}
\phi_{n}(\delta)&=\sup \left\{\left(P-P_{n}\right)\left(\eta(x,\cdot)-\eta\left(x^{*}, \cdot\right)\right): x \in \cM(\delta)\right\} \\
&=\sup \left\{\left(P-P_{n}\right)\left(\eta_c(x,\cdot)-\eta_c\left(x^{*},\cdot\right)\right): x \in \cM(\delta)\right\}
\end{align*}
for $\delta \geq 0$.
Since $x_{n}$ is a minimizer of $P_{n} \eta(x,\cdot)$, it follows from the definition of $\phi_{n}$ that
\begin{equation*}
\delta_{n} \leq\left(P-P_{n}\right)\left(\eta\left(x_{n},\cdot\right)-\eta\left(x^{*}, \cdot\right)\right) \leq \phi_{n}\left(\delta_{n}\right).
\end{equation*}
Applying \cref{Lemma1,Lemma2} we get that, with probability at least $1-(\Delta/2)$,
\begin{equation}\label{phibound}
\begin{split}
\phi_{n}(\delta) &\leq 4 \mathbb{E}\,\phi_{n}(\delta)+\frac{8\sqrt{2}\cdot \bar\sigma(\delta)}{\sqrt{n \Delta}}.
\end{split}
\end{equation}

We first get an upper bound to $\mathbb{E}\,\phi_{n}(\delta)$. Let $\{\varepsilon_{i}\}$ be a Rademacher sequence, i.e. random signs independent of $X_{i}$'s.
Then, by the symmetrization of the associated empirical process (see \cite{gine2021mathematical}) we obtain
\begin{equation*}
\begin{split}
\mathbb{E}\,\phi_{n}(\delta)
&\leq 2\,\mathbb{E} \,\bigg(\sup_{x\in \cM(\delta)} n^{-1}\sum_{i=1}^{n} \varepsilon_{i}\left(\eta_c(x,X_{i})-\eta_c(x^*,X_{i})\right)\bigg)  \\
&= 2\,\mathbb{E} \,\bigg(\sup_{x\in \cM(\delta)} n^{-1}\sum_{i=1}^n \varepsilon_{i}\,\eta_c(x,X_{i})\bigg).
\end{split}
\end{equation*}
One can easily check that the Rademacher empirical process $\{Y_{f} : f\in (\mathcal{F}(\delta), \|\cdot\|_{2,P_{n}}) \}$
for the pseudo metric space $(\mathcal{F}(\delta), \|\cdot\|_{2,P_{n}})$ given by
\begin{equation*}
Y_{f}:=\frac{1}{\sqrt{n}}\sum_{i=1}^{n} \varepsilon_{i}f(X_{i})
\end{equation*}
is $\sqrt{n}$-Lipschitz with respect to $\|\cdot\|_{2,P_{n}}$, conditionally on the $X_{i}$’s.
It is also sub-Gaussian. To see this, we note that, for any $a_{1},\dots,a_{n} \in \mathbb{R}$,
\begin{align*}
    \mathbb{E} \bigg(\exp \Big(\sum_{i=1}^{n} a_{i}\varepsilon_{i} \Big) \bigg)
    =\prod_{i=1}^{n} \mathbb{E}e^{a_{i}\varepsilon_{i}}
    =\prod_{i=1}^{n} \frac{e^{a_{i}}+e^{-a_{i}}}{2}
    \le \prod_{i=1}^{n} e^{a_{i}^{2}/2}
    =\exp \Big(\sum_{i=1}^{n} \frac{a_{i}^{2}}{2} \Big),
\end{align*}
where the inequality follows from Taylor's expansion. From this we get that, for any $f,g \in \mathcal{F}(\delta)$ and $\theta \in \mathbb{R}$,
\begin{align*}
    \mathbb{E}\big(e^{\theta\left(Y_{f}-Y_{g}\right)} | X_{1}, \dots, X_{n} \big)
    &=\mathbb{E} \bigg(\exp \Big( \frac{\theta}{\sqrt{n}} \sum_{i=1}^{n} \varepsilon_{i}(f-g)(X_{i}) \Big) \Big| X_{1}, \dots, X_{n} \bigg) \\
    &\leq\exp \bigg(\frac{\theta^{2}}{2n} \sum_{i=1}^{n} (f(X_{i})-g(X_{i}))^{2} \bigg)\\
    &=\exp \Big(\frac{\theta^{2}}{2} \|f-g\|_{2,P_{n}}^{2} \Big).
\end{align*}
Thus, $Y_{f}$ satisfies the conditions of \cref{Lemma3} (see \cite{ahidar2019rate}).
Applying \cref{Lemma3} with (B1) and using the inequalities for $H_{\delta}$ given in \cref{Lemma2}, we get
\begin{equation}\label{exptbound}
\begin{split}
\mathbb{E}\,\phi_{n}(\delta) &\leq 2\,\mathbb{E}\,\inf_{\varepsilon \geq 0}\left(2\varepsilon+\frac{12}{\sqrt{n}}
\int_{\varepsilon}^{\infty} \sqrt{\log N(u, \mathcal{F}(\delta),\|\cdot\|_{2,P_{n}})} \rd u \right) \\
&\leq 2\,\mathbb{E}\,\inf_{\varepsilon \geq 0}\left(2\varepsilon+\frac{12}{\sqrt{n}} \int_{\varepsilon}^{\|H_{\delta}\|_{2,P_{n}}} \sqrt{D\log \left(\frac{A\|H_{\delta}\|_{2,P_{n}}}{u}\right)} \,\rd u \right) \\
&= 2\, \mathbb{E}\left(\|H_{\delta}\|_{2,P_{n}}\right)\cdot \inf_{\varepsilon' \geq 0}\, \left(2\varepsilon'+ \frac{12}{\sqrt{n}} \int_{\varepsilon'}^{1} \sqrt{D\log \left(\frac{A}{u}\right)} \,\rd u \right)\\
&\leq 48 \,\mathbb{E}\left(\|H_{\delta}\|_{2,P_{n}}\right)\cdot \sqrt{\frac{AD}{n}}\\
&\leq 48\,\bar\sigma(\delta)\sqrt{\frac{AD}{n}},
\end{split}
\end{equation}
where in the third inequality we have used $\log x \leq x-1 \leq x$ for $x>0$.

The inequalities \eqref{phibound} and \eqref{exptbound} imply that, with probability at least $1-(\Delta/2)$,
\begin{equation*}
\begin{split}
\phi_{n}(\delta) &\leq \bar\sigma(\delta)\left( 192\sqrt{\frac{AD}{n}}+\frac{8\sqrt{2}}{\sqrt{n \Delta}} \right) \\
&\leq 32\sqrt{\frac{K\sigma_{X}^2 \delta^{\beta}}{n}} \left(24\sqrt{AD}+\sqrt{\frac{2}{\Delta}} \right)\\
&=:b_{n}(\delta,\Delta).
\end{split}
\end{equation*}
Since $\phi_{n}(\delta)$ is an increasing function and $b_{n}(\delta,\Delta)$ is decreasing in $\Delta$ for fixed $\delta$, it follows from Theorem 4.3 in \cite{koltchinskii2011oracle} that
\begin{equation}\label{estbound}
\delta_{n} \leq \phi_{n}(\delta_{n}) \leq b_{n}(\Delta):=\inf \left\{\tau>0: \sup_{\delta \geq \tau} \delta^{-1} b_{n}\left(\delta, \Delta^{\frac{\delta}{\tau}} \right) \leq 1 \right\}
\end{equation}
with probability at least $1-\Delta$. Since $\bar\sigma(\delta)/\delta$ is decreasing in $\delta$ as $\beta \in (0,2)$,
\begin{equation*}
\sup_{\delta \geq \tau} \delta^{-1} b_{n}\left(\delta, \Delta^{\frac{\delta}{\tau}} \right)=\frac{b_{n}(\tau,\Delta)}{\tau}=32\sqrt{\frac{K\sigma_{X}^2 \tau^{-(2-\beta)}}{n}} \left(24\sqrt{AD}+\sqrt{\frac{2}{\Delta}}\right).
\end{equation*}
This gives
\begin{align}\label{bNbound}
b_{n}(\Delta)
&=\inf \left\{\tau>0: 32\sqrt{\frac{K\sigma_{X}^2 \tau^{-(2-\beta)}}{n}} \left(24\sqrt{AD}+\sqrt{\frac{2}{\Delta}}\right) \leq 1 \right\} \\
&=\left\{32\sqrt{\frac{K\sigma_{X}^2}{n}} \left(24\sqrt{AD}+\sqrt{\frac{2}{\Delta}}\right) \right\}^{\frac{2}{2-\beta}}. \nonumber
\end{align}
Applying \eqref{bNbound} to \eqref{estbound}, we obtain that, with probability at least $1-\Delta$,
\begin{align*}
l(x_{n},x^*) &\leq \sqrt{K}\cdot \delta_{n}^{\beta/2} \\
&\leq K^{\frac{1}{2-\beta}} \left\{32 \left(24\sqrt{AD}+\sqrt{\frac{2}{\Delta}}\right) \frac{\sigma_{X}}{\sqrt{n}} \right\}^{\frac{\beta}{2-\beta}}.
\end{align*}
This completes the proof of \cref{thm_fin_conv}.
\end{proof}

\begin{proof}[Proof of \cref{thm_inf_conv}]
The proof is similar to that of \cref{thm_fin_conv} for the case of finite-dimensional $\cM$. The difference is in the covering number $N(u, \mathcal{F}(\delta),\|\cdot\|_{2,P_{n}})$.
We get
\begin{align*}
\mathbb{E}\,\phi_{n}(\delta)
&\leq 2\,\mathbb{E}\,\inf_{\varepsilon \geq 0}\left(2\varepsilon+\frac{12}{\sqrt{n}} \int_{\varepsilon}^{\|H_{\delta}\|_{2,P_{n}}}
\sqrt{\frac{A \|H_{\delta}\|_{2,P_{n}}^{2\zeta}}{u^{2\zeta}}} \,\rd u \right) \\
&= 4\,\mathbb{E} \,\left(\|H_{\delta}\|_{2,P_{n}}\right)\cdot \inf_{\varepsilon \geq 0}\, \left(\varepsilon+ 6\sqrt{\frac{A}{n}} \int_{\varepsilon}^{1} u^{-\zeta}\, \rd u \right)\\
&\le 4\,\mathbb{E} \,\left(\|H_{\delta}\|_{2,P_{n}}\right)\times
\begin{cases}
\displaystyle \frac{6}{1-\zeta}\sqrt{\frac{A}{n}} \quad &\text{if }0<\zeta<1\\[4mm]
\displaystyle 6\sqrt{\frac{A}{n}}\left(1-\log\Big(6\sqrt{\frac{A}{n}}\Big)\right) \quad &\text{if }\zeta=1\\[2mm]
\displaystyle \frac{\zeta}{\zeta-1}\left(6\sqrt{\frac{A}{n}}\right)^{1/\zeta} \quad &\text{if }\zeta>1.
\end{cases}.
\end{align*}
Therefore, $\phi_{n}(\delta_{n})\le b_{n}(\Delta)$ with probability at least $1-\Delta$, now with
\begin{align*}
b_{n}(\Delta)=
\begin{cases}
\displaystyle \left(32\sqrt{K\sigma_{X}^2} \left(\frac{12}{1-\zeta}  \sqrt{\frac{A}{n}} +\sqrt{\frac{2}{n \Delta}}\right) \right)^{\frac{2}{2-\beta}}\quad &\text{if }0<\zeta < 1\\
\displaystyle \left(32\sqrt{ K\sigma_{X}^2} \left(12 \sqrt{\frac{A}{n}} \left(1-\log \left(6\sqrt{\frac{A}{n}}\right) \right) +\sqrt{\frac{2}{n \Delta}}\right) \right)^{\frac{2}{2-\beta}}\quad &\text{if }\zeta = 1\\
\displaystyle \left(32\sqrt{K\sigma_{X}^2} \left(\frac{2\zeta}{\zeta-1} \left(6 \sqrt{\frac{A}{n}}\right)^{1/\zeta} +\sqrt{\frac{2}{n \Delta}}\right) \right)^{\frac{2}{2-\beta}}\quad &\text{if }\zeta>1.
\end{cases}
\end{align*}
This gives the theorem.
\end{proof}

\subsection{Proofs of theorems in \cref{Sec_MoM}}

Without loss of generality, we assume that $n=m\cdot k$, where $k$ is the number of
blocks in splitting the sample and $m$ is the size of each block.
\begin{proof}[Proof of \cref{thm_fin_mom}]
Let $F(x)= \int_{\mathcal{M}} \eta(x,y)\,\mathrm{d} P(y)$.
By the definition of $x^*$ it holds that, for each block $\mathcal{B}_j$,
\begin{equation*}
    F_{n,j}(x^*)-F_{n,j}(Z_{j}) \leq F_{n,j}(x^*)-F_{n,j}(Z_{j})-F(x^*)+F(Z_{j}).
\end{equation*}
The right hand side has an upper bound that is analogous to $\phi_{n}(\delta_{n})$ in the proof of \cref{thm_fin_conv},
which is obtained by substituting the empirical measure corresponding to $\mathcal{B}_j$ for $P_n$ and $Z_j$ for $x_n$.
Thus, replacing $\Delta$ by  $(1-2q)/2$ (so that $1-\Delta$ by $q+1/2$)
and $n$ by $m=n/k$ with $K=\beta=1$, we get from \eqref{estbound} and \eqref{bNbound} that
\begin{equation}\label{Bernoulli-prob}
\mathbb{P} \left(F_{n,j}(x^*)-F_{n,j}(Z_{j}) \leq \varepsilon_{q}^{2}\right) \geq q+\frac{1}{2},
\end{equation}
where
\begin{align}\label{ekq-fin}
\varepsilon_{q}=32\sqrt{\frac{k\sigma_{X}^2}{n}} \left(24\sqrt{AD}+\frac{2}{\sqrt{1-2q}} \right).
\end{align}
By the CN inequality in \cref{Subsection 4.1}, we have
\begin{equation*}
\begin{split}
&F_{n,j}(Z_{j}) \leq F_{n,j}(\gamma_{1/2}^x) \leq \frac{F_{n,j}(x)}{2}+\frac{F_{n,j}(x^*)}{2}-\frac{d(x^*,x)^{2}}{4}\\
\Leftrightarrow\; &F_{n,j}(x)-F_{n,j}(Z_{j}) \geq -\left(F_{n,j}(x^*)-F_{n,j}(Z_{j})\right)+\frac{d(x^*,x)^2}{2},
\end{split}
\end{equation*}
where $\gamma^x:[0,1]\to\cM$ is the geodesic with $\gamma_0^x=x^*$ and $\gamma_1^x=x$.
Thus, denoting by $\cE_{n,j}$ the event
\[
F_{n,j}(x)> F_{n,j}(x^*) \quad\mbox{for all }\,x \in \mathcal{M} \mbox{ with } d(x,x^*)>2\varepsilon_{q},
\]
we get from (\ref{Bernoulli-prob}) that $\mathbb{P}\,(\cE_{n,j})\ge q+1/2$ since $F_{n,j}(x^*)-F_{n,j}(Z_{j}) \leq \varepsilon_{q}^{2}$ implies
\[
F_{n,j}(x)-F_{n,j}(Z_{j}) > -\left(F_{n,j}(x^*)-F_{n,j}(Z_{j})\right)+2\varepsilon_{q}^2 \;\geq \;F_{n,j}(x^*)-F_{n,j}(Z_{j})
\]
for all $x$ with $d(x,x^*)>2\varepsilon_{q}$. By applying H\o ffding's inequality to $\sum_{j=1}^k I(\cE_{n,j})$, we obtain
\begin{align*}
1-\Delta &\le 1-e^{-2q^2k}\\
&\le \mathbb{P}\left(\sum_{j=1}^k I(\cE_{n,j})>k/2 \right)\\
&\le \mathbb{P}\left(\sum_{j=1}^k I\left(F_{n,j}(x)> F_{n,j}(x^*)\right)>k/2 \;\mbox{ for all } x\in \cM \text{ with } d(x,x^*)>2\varepsilon_{q}\right).
\end{align*}
This completes the proof of the theorem.
\end{proof}

\begin{proof}[Proof of \cref{thm_inf_mom}]
The proof is essentially the same as that of \cref{thm_fin_mom} except that we use \cref{thm_inf_conv} instead of \cref{thm_fin_conv}.
We obtain (\ref{Bernoulli-prob}) now with
\begin{align}\label{ekq-inf}
\varepsilon_{q,\zeta}=
\begin{cases}
\displaystyle C_{A,1} \cdot \frac{\log (n/k)}{\sqrt{n/k}} \cdot \frac{\sigma_{X}}{\sqrt{(1-2q)/2}}, &\text{if}\;\;\zeta = 1\\ 
\quad \\
\displaystyle C_{A,\zeta}\cdot (k/n)^{1/2\zeta} \cdot \frac{\sigma_{X}}{\sqrt{(1-2q)/2}}, &\text{if}\;\;\zeta>1. 
\end{cases}
\end{align}
Since
\begin{equation}\label{logapprox}
\begin{split}
   &\frac{1}{\sqrt{2}q}\frac{\log n \sqrt{\log(1/\Delta)}}{\sqrt{n}}=\frac{\sqrt{k} \log n}{\sqrt{n}} \geq \frac{\log (n/k)}{\sqrt{n/k}}, \\
   &\frac{1}{\sqrt{2}q} n^{-1/2\zeta} \left(\log \frac{1}{\Delta}\right)^{1/2\zeta} \geq \left(\frac{2q^{2}n}{\log(1/\Delta)}\right)^{-1/2\zeta}=(k/n)^{1/2\zeta},
\end{split}
\end{equation}
we get $\varepsilon_{q,\zeta} \le R_{q,\zeta}/2$.
The rest of the proof is the same as in the proof of \cref{thm_fin_mom}.
\end{proof}

\begin{proof}[Proof of \cref{mom_breakdown}]
Fix a partition $\{\mathcal{B}_j: 1 \le j \le k\}$ of the original dataset $\cX_n=\{X_1, \ldots, X_n\}$. Let $\mathcal{B}_j=\{X_{1,j}, \ldots, X_{m,j}\}$ for $1\le j\le k$.
Choose a point $O \in \cM$ and let $D_{O}=\max_{1 \le i \le n} d(O,X_{i})$.
We let $\tilde X_i$ denote a corrupted value corresponding to $X_i$, and we write $\tilde A$ instead of $A$ if a term $A$ involves corrupted values.
For example, we write
$\tilde F_{n,j}(x)$ instead of $F_{n,j}(x)$ when the $j$th block contains a corrupted value.
In particular, we write in this proof $\tilde S_x$ and $\tilde r_x$ for each point $x\in \cM$
rather than $S_x$ and $r_x$, respectively, since the defeating region and radius always depend on corrupted values.
Put $L:=\lceil (k+1)/2 \rceil$.

We first show that $\varepsilon_{n}^* \le L/n$ by contradiction. Suppose that
it is false, i.e., $\varepsilon_{n}^* > L/n$. Then, for an arbitrary configuration $\{j(1),\ldots,j(L)\} \subset \{1,2,\ldots,k\}$ with the corruption $\tilde X_{1,1}=\tilde X_{1,2}=\dots=\tilde X_{1,L}=\tilde x$,
there exists $R>0$ such that $\sup_{\tilde x \in \cM} d(O, \tilde x_{MM}) < R$. We may assume $R>\sqrt{m-1} \cdot D_{O}$.
Now, let $\tilde \gamma:[0,1]\rightarrow \mathcal{M}$ be the geodesic connecting $\tilde \gamma_0=O$ and $\tilde \gamma_1=\tilde x$ with the length $\tilde D:=d(\tilde x,O)$ larger than $R$.
For $j=1,\dots,L$, by the triangular inequality,
\begin{align*}
    m \cdot (\tilde F_{n,j}(\tilde x_{MM})-\tilde F_{n,j}(\tilde \gamma_{t}))
    &\ge d(\tilde x,\tilde x_{MM})^{2}-(1-t)^{2}\tilde D^{2}-\sum_{i=2}^{m} d(\tilde \gamma_{t},X_{ij})^{2} \\
    &\ge (\tilde D-R)^{2}-(1-t)^{2}\tilde D^{2}-(m-1)(t\tilde D+D_{O})^{2} \\
    &= (2t-mt^{2})\tilde D^{2}-2(R+D_{O}(m-1)t)\tilde D+(R^{2}-(m-1) \cdot D_{O}^{2}) \\
    &\ge (2t-mt^{2})\tilde D^{2}-2(R+D_{O}(m-1)t)\tilde D.
\end{align*}
This implies that for all $t<2/m$, the point $\tilde \gamma_{t}$ defeats $\tilde x_{MM}$ whenever
\begin{equation*}
    \tilde D \ge \frac{2(R+D_O(m-1)t)}{2t-mt^{2}},
\end{equation*}
so the defeating radius of $\tilde x_{MM}$ satisfies $\tilde r_{\tilde x_{MM}} \ge d(\tilde \gamma_{t},\tilde x_{MM})$.
Therefore,
\begin{equation}\label{radi_corrupt}
    \liminf_{\tilde D \rightarrow \infty} \frac{\tilde r_{\tilde x_{MM}}}{\tilde D}
    \ge \liminf_{\tilde D \rightarrow \infty} \sup_{0<t<2/m} \frac{d(\tilde \gamma_{t},\tilde x_{MM})}{\tilde D}
    = \liminf_{\tilde D \rightarrow \infty} \sup_{0<t<2/m} \frac{d(\tilde \gamma_{t},O)}{\tilde D}
    = \frac{2}{m}.
\end{equation}
Now, choose any $x \in \tilde S_{\tilde \gamma_{1/m}} (\neq \emptyset)$, i.e. $x$ defeating $\tilde \gamma_{1/m}$. Then, $\exists$ at least one $j_{0} \in \{j(1), \ldots, j(L)\}$ such that $\tilde F_{n,j_{0}}(x) \le \tilde F_{n,j_{0}}(\tilde \gamma_{1/m})$. Due to the CN inequality,
\begin{align}\label{CN_breakdown}
    0 &\le m \cdot (\tilde F_{n,j_{0}}(\tilde \gamma_{1/m})-\tilde F_{n,j_{0}}(x)) \\
    &= d(\tilde \gamma_{1/m},\tilde x)^{2}-d(x,\tilde x)^{2}+ \sum_{i=2}^{m} \left( d(\tilde \gamma_{1/m},X_{i,j_{0}})^{2} - d(x,X_{i,j_{0}})^{2} \right) \nonumber\\
    &\le \left(\frac{m-1}{m} \right)^{2} \tilde D^{2}-d(x,\tilde x)^{2} \nonumber \\
    &\qquad+ \sum_{i=2}^{m} \left\{ \frac{m-1}{m} d(O,X_{i,j_{0}})^{2} +\frac{1}{m}d(\tilde x,X_{i, j_{0}})^{2} - \frac{m-1}{m^{2}} \tilde D^{2} - d(x,X_{i,j_{0}})^{2} \right\}.
    \nonumber
\end{align}
Note that again by the CN inequality,
\begin{equation*}
    d(x,\tilde \gamma_{1/m})^{2} \le \frac{m-1}{m} d(O,x)^{2}+\frac{1}{m} d(x,\tilde x)^{2}-\frac{m-1}{m^{2}} \tilde D^{2}.
\end{equation*}
Plugging this inequality into \eqref{CN_breakdown} and using the triangular inequality, we get
\begin{align*}
    0
    &\le \frac{m-1}{m^{2}} \tilde D^{2} -m \cdot d(x,\tilde \gamma_{1/m})^{2}\\
    &\quad + \sum_{i=2}^{m} \left\{ \frac{m-1}{m} d(O,X_{i,j_{0}})^{2} +\frac{1}{m}d(\tilde x,X_{i, j_{0}})^{2} - \frac{m-1}{m^{2}} \tilde D^{2} +d(O,x)^{2}- d(x,X_{i,j_{0}})^{2} \right\} \\
    &\le \frac{m-1}{m^{2}} \tilde D^{2} -m \cdot d(x,\tilde \gamma_{1/m})^{2}\\
    &\quad + (m-1) \left\{ \frac{m-1}{m} D_{O}^{2} +\frac{1}{m}(\tilde D+D_{O})^{2} - \frac{m-1}{m^{2}} \tilde D^{2} +2D_{O} \cdot d(O,x)- D_{O}^{2} \right\} \\
    &\le \frac{m-1}{m^{2}} \tilde D^{2} -m \cdot d(x,\tilde \gamma_{1/m})^{2}\\
    &\quad + (m-1) \left\{ \frac{m-1}{m} D_{O}^{2} +\frac{1}{m}(\tilde D+D_{O})^{2} - \frac{m-1}{m^{2}} \tilde D^{2} +2D_{O} \left(\frac{\tilde D}{m}+d(x, \tilde \gamma_{1/m}) \right)- D_{O}^{2} \right\} \\
    &=-m \cdot d(x,\tilde \gamma_{1/m})^{2}+2(m-1)D_{O} \cdot d(x,\tilde \gamma_{1/m})+\frac{2(m-1)\tilde D(\tilde D+2mD_{O})}{m^{2}}.
\end{align*}
Therefore,
\begin{align*}
    d(x,\tilde \gamma_{1/m}) \le
    \frac{m(m-1)D_{O}+\sqrt{m^{2}(m-1)^{2}D_{O}^{2}+2m(m-1)\tilde D(\tilde D+2mD_{O})} }{m^{2}}.
\end{align*}
Since $x \in \tilde S_{\tilde \gamma_{1/m}}$ was chosen arbitrarily, we have
\begin{align*}
    \limsup_{\tilde D \rightarrow \infty} \frac{\tilde r_{\tilde \gamma_{1/m}}}{\tilde D}
    &\le \limsup_{\tilde D \rightarrow \infty} \frac{m(m-1)D_{O}+\sqrt{m^{2}(m-1)^{2}D_{O}^{2}+2m(m-1)\tilde D(\tilde D+2mD_{O})} }{m^{2} \tilde D} \\
    &= \frac{\sqrt{2}}{m} \\
    &< \frac{2}{m}.
\end{align*}
In view of \eqref{radi_corrupt}, the above strict inequality is contradictory to the fact
that $\tilde r_{\tilde x_{MM}}\le \tilde r_{\tilde \gamma_{1/m}}$ for all $\tilde D$ from the definition of geometric-median-of-means.

Next, we show that
\begin{equation}\label{finite-radius}
\mbox{$\sup_*$} d(O, \tilde x_{MM})<\infty,
\end{equation}
where $\sup_*$ denotes the supremum over all configurations of $s \le (L-1)$ arbitrary corruptions among $\{X_1, \ldots, X_n\}$.
We note that \eqref{finite-radius} implies $\varepsilon_{n}^* \ge L/n$.
To prove \eqref{finite-radius}, let $s \le L-1$ be the number of corrupted $\tilde X_i$ and think of a configuration of the indices of $\tilde X_i$,
say $\{i(1), \ldots, i(s)\} \subset \{1,2,\ldots,n\}$. The corrupted $\tilde X_{i(1)}, \ldots, \tilde X_{i(s)}$ are scattered across the $k$ blocks $\cB_j,\, 1 \le j \le k$.
Without loss of generality, let $\cB_1, \ldots, \cB_{J}$ denote those blocks that do not contain any of the corrupted values. We note that $J>k/2$ since $s \le L-1$.
We claim
\begin{equation}\label{finite-radius-1}
    \sup_{\tilde X_{i(1)}, \ldots, \tilde X_{i(s)}}\max_{1 \le j \le J} \tilde r_{Z_{j}} < \infty.
\end{equation}
Then, by the definition of the geometric-median-of-means we get
\begin{equation}\label{finite-radius-2}
\sup_{\tilde X_{i(1)}, \ldots, \tilde X_{i(s)}}\tilde r_{\tilde x_{MM}} < \infty.
\end{equation}
Also, by the definition of $x$-defeating radius and since $J>k/2$,
it holds that
\begin{align}\label{radius-lower-bound}\begin{split}
\tilde r_{x} &\ge {\rm rad}_x \left(\bigcap_{j=1}^{J} \{y \in \cM: F_{n,j}(y) \le F_{n,j}(x) \} \right)\\
&\ge {\rm rad}_x \left( \bigcap_{j=1}^k \{y \in \cM: F_{n,j}(y) \le F_{n,j}(x) \} \right)
\end{split}
\end{align}
for all $x \in \cM$, where ${\rm rad}_x (A)$ stands for the radius of the smallest ball centered at $x$ that covers $A$.
The right hand side of the second inequality in \eqref{radius-lower-bound} depends solely on the original dataset $\{X_1, \ldots, X_n\}$,
independent of data corruption.
Now, suppose that there exists $s \le L-1$ and a configuration $\{i(1), \ldots, i(s)\}$
such that
\[
\sup_{\tilde X_{i(1)}, \ldots, \tilde X_{i(s)}}d(O,\tilde x_{MM})=\infty.
\]
Then, since the right hand side of the second inequality in \eqref{radius-lower-bound} diverges to infinity as $d(O,x) \rightarrow \infty$,
we would obtain
\[
\sup_{\tilde X_{i(1)}, \ldots, \tilde X_{i(s)}}\tilde r_{\tilde x_{MM}}=\infty,
\]
which contradicts \eqref{finite-radius-2}. This proves \eqref{finite-radius}.

It remains to prove \eqref{finite-radius-1}.
Let $1 \le j \le J$ be fixed. Then, for any $x$ that defeats $Z_{j}$,
there exists at least one un-corrupted block $\cB_l$ ($1 \le l \le J$) such that $F_{n,l}(x) \le F_{n,l}(Z_{j})$, since $J>k/2$ and
the number of indices $l:1\le l\le k$ such that $F_{n,l}(x) \le F_{n,l}(Z_{j})$ or $\tilde F_{n,l}(x) \le \tilde F_{n,l}(Z_{j})$ is greater than $k/2$.
This implies that
\begin{align}\label{finite-radius-3}\begin{split}
\tilde r_{Z_j}&\le \max_{x \in \cM}\{d(x,Z_j): F_{n,l}(x) \le F_{n,l}(Z_{j})\mbox{ for some }1 \le l \le J \} \\
    &\le \max_{x \in \cM}\Big\{\sqrt{F_{n,l}(x)} + \max_{X_i \in \cB_l}d(X_{i},Z_{j}): F_{n,l}(x) \le F_{n,l}(Z_{j})\mbox{ for some }1 \le l \le J \Big\} \\
    &\le \max_{1\le l \le J} \Big(\sqrt{F_{n,l}(Z_j)} + \max_{X_i \in \cB_l}d(X_{i},Z_{j})\Big)\\
    &\le \max_{1 \le l \le k} \Big(\sqrt{F_{n,l}(Z_{j})} + \max_{X_i \in \cB_l}d(X_{i},Z_{j})\Big).
\end{split}
\end{align}
In \eqref{finite-radius-3}, the second inequality follows from
\begin{align*}
d(x, Z_{j})
\le \frac{1}{m}\sum_{X_{i} \in \mathcal{B}_{l}} \left( d(x,X_{i})+d(X_{i},Z_{j}) \right)
\le \sqrt{F_{n,l}(x)}+\max_{X_{i} \in \mathcal{B}_{l}} d(X_{i},Z_{j}).
\end{align*}

The right hand side of the last inequality in \eqref{finite-radius-3} depends solely on the original dataset $\{X_1, \ldots, X_n\}$,
independent of the configuration of $\{i(1), \ldots, i(s)\}$
and the corrupted values $\tilde X_{i(1)}, \ldots, \tilde X_{i(s)}$.
This gives \eqref{finite-radius-1}.
\end{proof}

\begin{proof}[Proof of \cref{thm_fin_power_mom}]
First, we follow the lines leading to (\ref{Bernoulli-prob}), now
using \eqref{bNbound} with $K=K_{\alpha}$ and $\beta=2-2/\alpha$ instead of $K=\beta=1$. We may prove
\begin{equation}\label{Bernoulli-prob-power}
\mathbb{P} \left(F_{n,j}(x^*)-F_{n,j}(Z_{j}) \leq K_{\alpha}^{\alpha/2}\varepsilon_{q}^\alpha\right) \geq q+\frac{1}{2}.
\end{equation}
By integrating both sides of the inequality in \cref{thm_gen_CN} with respect to $z$
for $\gamma=\gamma^x:[0,1]\to \cM$,
we obtain that, for all $0 \leq t \leq 1$ and $\delta>0$,
\begin{align*}
&(1+\delta)^{1-\alpha/2} \left((1-t)^{\alpha/2} F_{n,j}(x^*)+t^{\alpha/2}F_{n,j}(x) \right)- F_{n,j}(\gamma_{t}^x)\\
&\hspace{1cm} \geq \delta^{1-\alpha/2} \left(t(1-t)d(x, x^*)^{2}\right)^{\alpha/2}.
\end{align*}
From the definition of $Z_{j}$ and the above inequality, we get
\begin{align*}
F_{n,j}(Z_{j}) \leq F_{n,j}(\gamma_{t})
&\leq (1+\delta)^{1-\alpha/2} \left((1-t)^{\alpha/2} F_{n,j}(x^*)+t^{\alpha/2}F_{n,j}(x) \right)\\
&\qquad -\delta^{1-\alpha/2} \left(t(1-t)d(x, x^*)^{2}\right)^{\alpha/2}.
\end{align*}
This gives that, on the event where $F_{n,j}(x^*)-F_{n,j}(Z_{j}) \leq K_{\alpha}^{\alpha/2}\varepsilon_{q}^\alpha$,
\begin{align*}
&(1+\delta)^{1-\alpha/2} t^{\alpha/2}F_{n,j}(x)\\
&\quad > \left(1-(1+\delta)^{1-\alpha/2} (1-t)^{\alpha/2}\right) F_{n,j}(x^*)
+\left(\delta^{1-\alpha/2} t^{\alpha/2} (1-t)^{\alpha/2}-M_{\alpha,\rho}\right) \cdot \frac{K_{\alpha}^{\alpha/2} \varepsilon_{q}^{\alpha}}{M_{\alpha,\rho}}
\end{align*}
or equivalently
\begin{align*}
F_{n,j}(x) > \frac{1-(1+\delta)^{1-\frac{\alpha}{2}} (1-t)^{\frac{\alpha}{2}}}{(1+\delta)^{1-\frac{\alpha}{2}} t^{\frac{\alpha}{2}}}\cdot F_{n,j}(x^*) +\frac{\delta^{1-\frac{\alpha}{2}} t^{\frac{\alpha}{2}} (1-t)^{\frac{\alpha}{2}}-M_{\alpha,\rho}}{(1+\delta)^{1-\frac{\alpha}{2}} t^{\frac{\alpha}{2}}} \cdot \frac{K_{\alpha}^{\alpha/2} \varepsilon_{q}^{\alpha}}{M_{\alpha,\rho}}
\end{align*}
for all $x\in \cM$ with $d(x,x^*) > K_{\alpha}^{1/2} M_{\alpha,\rho}^{-1/\alpha}\varepsilon_{q}$. Thus, from \eqref{Bernoulli-prob-power}
and the definition of $M_{\alpha,\rho}$ it follows that
\begin{equation}\label{Bern-prob-power}
\mathbb{P}\left(F_{n,j}(x) > \rho \cdot F_{n,j}(x^*) \,\mbox{ for all }\,
x \in \cM \mbox{ with } d(x,x^*)>\frac{K_{\alpha}^{1/2}\varepsilon_{q}}{M_{\alpha,\rho}^{1/\alpha}}\right)\ge q+\frac{1}{2}.
\end{equation}
Applying H\o ffding's inequality as in the proof of \cref{thm_fin_mom} with \eqref{Bern-prob-power}, we may complete the proof of the theorem.
\end{proof}

\begin{proof}[Proof of \cref{thm_inf_power_mom}]

The proof is essentially the same as that of \cref{thm_fin_power_mom} except that we use the definition of $\varepsilon_{q,\zeta}$ at \eqref{ekq-inf} instead of $\varepsilon_{q}$
at \eqref{ekq-fin}. Using \eqref{logapprox} we get $K_{\alpha}^{1/2} M_{\alpha,\rho}^{-1/\alpha}\varepsilon_{q,\zeta} \le R_{q,\alpha,\rho,\zeta}$.
\end{proof}

\end{appendix}

\begin{acks}[Acknowledgments]
The authors would like to thank an associate editor and three referees for constructive comments.
\end{acks}

\begin{funding}
Research of the authors was supported by the National Research Foundation of Korea (NRF) grant funded by
the Korea government (MSIP) (No. 2019R1A2C3007355).
\end{funding}

\begin{supplement}

The Supplementary Material contains an additional proposition and proofs of the propositions in \cref{Section 4}.
\end{supplement}

\bibliographystyle{imsart-number}
\bibliography{mybib-revised}       

\begin{thebibliography}{54}

\bibitem{adamczak2008tail}
\begin{barticle}[author]
\bauthor{\bsnm{Adamczak},~\bfnm{Radoslaw}\binits{R.}}
(\byear{2008}).
\btitle{A tail inequality for suprema of unbounded empirical processes with
  applications to Markov chains}.
\bjournal{Electronic Journal of Probability}
\bvolume{13}
\bpages{1000--1034}.
\end{barticle}
\endbibitem

\bibitem{ahidar2019rate}
\begin{barticle}[author]
\bauthor{\bsnm{Ahidar-Coutrix},~\bfnm{Adil}\binits{A.}},
  \bauthor{\bsnm{Le~Gouic},~\bfnm{Thibaut}\binits{T.}} \AND
  \bauthor{\bsnm{Paris},~\bfnm{Quentin}\binits{Q.}}
(\byear{2020}).
\btitle{Convergence rates for empirical barycenters in metric spaces:
  curvature, convexity and extendable geodesics}.
\bjournal{Probability Theory and Related Fields}
\bvolume{177}
\bpages{323--368}.
\end{barticle}
\endbibitem

\bibitem{Arsigny2007}
\begin{barticle}[author]
\bauthor{\bsnm{Arsigny},~\bfnm{V.}\binits{V.}},
  \bauthor{\bsnm{Fillard},~\bfnm{P}\binits{P.}},
  \bauthor{\bsnm{Pennec},~\bfnm{X}\binits{X.}} \AND
  \bauthor{\bsnm{Ayache},~\bfnm{N.}\binits{N.}}
(\byear{2007}).
\btitle{Geometric means in a novel vector space structure on symmetric
  positive-definite matrices}.
\bjournal{SIAM Journal of Matrix Analysis and Applications}
\bvolume{29}
\bpages{328--347}.
\end{barticle}
\endbibitem

\bibitem{bacak2014computing}
\begin{barticle}[author]
\bauthor{\bsnm{Ba{\u c}{\'a}k},~\bfnm{Miroslav}\binits{M.}}
(\byear{2014}).
\btitle{Computing medians and means in Hadamard spaces}.
\bjournal{SIAM Journal on Optimization}
\bvolume{24}
\bpages{1542--1566}.
\end{barticle}
\endbibitem

\bibitem{bacak2014convex}
\begin{bbook}[author]
\bauthor{\bsnm{Ba{\u c}{\'a}k},~\bfnm{Miroslav}\binits{M.}}
(\byear{2014}).
\btitle{Convex Analysis and Optimization in Hadamard Spaces}.
\bpublisher{de Gruyter}.
\end{bbook}
\endbibitem

\bibitem{bacak2018old}
\begin{barticle}[author]
\bauthor{\bsnm{Ba{\u c}{\'a}k},~\bfnm{Miroslav}\binits{M.}}
(\byear{2018}).
\btitle{Old and new challenges in Hadamard spaces}.
\bjournal{arXiv preprint arXiv:1807.01355}.
\end{barticle}
\endbibitem

\bibitem{bhattacharya2003large}
\begin{barticle}[author]
\bauthor{\bsnm{Bhattacharya},~\bfnm{Rabi}\binits{R.}} \AND
  \bauthor{\bsnm{Patrangenaru},~\bfnm{Vic}\binits{V.}}
(\byear{2003}).
\btitle{Large sample theory of intrinsic and extrinsic sample means on
  manifolds}.
\bjournal{The Annals of Statistics}
\bvolume{31}
\bpages{1--29}.
\end{barticle}
\endbibitem

\bibitem{bhattacharya2005large}
\begin{barticle}[author]
\bauthor{\bsnm{Bhattacharya},~\bfnm{Rabi}\binits{R.}} \AND
  \bauthor{\bsnm{Patrangenaru},~\bfnm{Vic}\binits{V.}}
(\byear{2005}).
\btitle{Large sample theory of intrinsic and extrinsic sample means on
  manifolds: {II}}.
\bjournal{The Annals of Statistics}
\bvolume{33}
\bpages{1225--1259}.
\end{barticle}
\endbibitem

\bibitem{billera2001geometry}
\begin{barticle}[author]
\bauthor{\bsnm{Billera},~\bfnm{Louis~J}\binits{L.~J.}},
  \bauthor{\bsnm{Holmes},~\bfnm{Susan~P}\binits{S.~P.}} \AND
  \bauthor{\bsnm{Vogtmann},~\bfnm{Karen}\binits{K.}}
(\byear{2001}).
\btitle{Geometry of the space of phylogenetic trees}.
\bjournal{Advances in Applied Mathematics}
\bvolume{27}
\bpages{733--767}.
\end{barticle}
\endbibitem

\bibitem{bolley2007quantitative}
\begin{barticle}[author]
\bauthor{\bsnm{Bolley},~\bfnm{Fran{\c{c}}ois}\binits{F.}},
  \bauthor{\bsnm{Guillin},~\bfnm{Arnaud}\binits{A.}} \AND
  \bauthor{\bsnm{Villani},~\bfnm{C{\'e}dric}\binits{C.}}
(\byear{2007}).
\btitle{Quantitative concentration inequalities for empirical measures on
  non-compact spaces}.
\bjournal{Probability Theory and Related Fields}
\bvolume{137}
\bpages{541--593}.
\end{barticle}
\endbibitem

\bibitem{boucheron2013concentration}
\begin{bbook}[author]
\bauthor{\bsnm{Boucheron},~\bfnm{St{\'e}phane}\binits{S.}},
  \bauthor{\bsnm{Lugosi},~\bfnm{G{\'a}bor}\binits{G.}} \AND
  \bauthor{\bsnm{Massart},~\bfnm{Pascal}\binits{P.}}
(\byear{2013}).
\btitle{Concentration inequalities: A nonasymptotic theory of independence}.
\bpublisher{Oxford university press}.
\end{bbook}
\endbibitem

\bibitem{catoni2012challenging}
\begin{barticle}[author]
\bauthor{\bsnm{Catoni},~\bfnm{Olivier}\binits{O.}}
(\byear{2012}).
\btitle{Challenging the empirical mean and empirical variance: a deviation
  study}.
\bjournal{Annales de l'IHP Probabilit{\'e}s et statistiques}
\bvolume{48}
\bpages{1148--1185}.
\end{barticle}
\endbibitem

\bibitem{catoni2018dimension}
\begin{barticle}[author]
\bauthor{\bsnm{Catoni},~\bfnm{Olivier}\binits{O.}} \AND
  \bauthor{\bsnm{Giulini},~\bfnm{Ilaria}\binits{I.}}
(\byear{2018}).
\btitle{Dimension-free PAC-Bayesian bounds for the estimation of the mean of a
  random vector}.
\bjournal{arXiv preprint arXiv:1802.04308}.
\end{barticle}
\endbibitem

\bibitem{chen2021wasserstein}
\begin{barticle}[author]
\bauthor{\bsnm{Chen},~\bfnm{Yaqing}\binits{Y.}},
  \bauthor{\bsnm{Lin},~\bfnm{Zhenhua}\binits{Z.}} \AND
  \bauthor{\bsnm{M{\"u}ller},~\bfnm{Hans-Georg}\binits{H.-G.}}
(\byear{2021}).
\btitle{Wasserstein regression}.
\bjournal{Journal of the American Statistical Association}
\bvolume{116}
\bpages{1--14}.
\end{barticle}
\endbibitem

\bibitem{cherapanamjeri2019fast}
\begin{binproceedings}[author]
\bauthor{\bsnm{Cherapanamjeri},~\bfnm{Yeshwanth}\binits{Y.}},
  \bauthor{\bsnm{Flammarion},~\bfnm{Nicolas}\binits{N.}} \AND
  \bauthor{\bsnm{Bartlett},~\bfnm{Peter~L}\binits{P.~L.}}
(\byear{2019}).
\btitle{Fast mean estimation with sub-Gaussian rates}.
In \bbooktitle{Conference on Learning Theory}
\bpages{786--806}.
\bpublisher{PMLR}.
\end{binproceedings}
\endbibitem

\bibitem{depersin2021robustness}
\begin{barticle}[author]
\bauthor{\bsnm{Depersin},~\bfnm{Jules}\binits{J.}} \AND
  \bauthor{\bsnm{Lecu{\'e}},~\bfnm{Guillaume}\binits{G.}}
(\byear{2021}).
\btitle{On the robustness to adversarial corruption and to heavy-tailed data of
  the Stahel-Donoho median of means}.
\bjournal{arXiv preprint arXiv:2101.09117}.
\end{barticle}
\endbibitem

\bibitem{depersin2022robust}
\begin{barticle}[author]
\bauthor{\bsnm{Depersin},~\bfnm{Jules}\binits{J.}} \AND
  \bauthor{\bsnm{Lecu{\'e}},~\bfnm{Guillaume}\binits{G.}}
(\byear{2022}).
\btitle{Robust sub-Gaussian estimation of a mean vector in nearly linear time}.
\bjournal{The Annals of Statistics}
\bvolume{50}
\bpages{511--536}.
\end{barticle}
\endbibitem

\bibitem{devroye2016sub}
\begin{barticle}[author]
\bauthor{\bsnm{Devroye},~\bfnm{Luc}\binits{L.}},
  \bauthor{\bsnm{Lerasle},~\bfnm{Matthieu}\binits{M.}},
  \bauthor{\bsnm{Lugosi},~\bfnm{Gabor}\binits{G.}} \AND
  \bauthor{\bsnm{Oliveira},~\bfnm{Roberto~I.}\binits{R.~I.}}
(\byear{2016}).
\btitle{Sub-Gaussian mean estimators}.
\bjournal{The Annals of Statistics}
\bvolume{44}
\bpages{2695--2725}.
\end{barticle}
\endbibitem

\bibitem{fernique1975regularite}
\begin{bincollection}[author]
\bauthor{\bsnm{Fernique},~\bfnm{Xavier}\binits{X.}}
(\byear{1975}).
\btitle{Regularit{\'e} des trajectoires des fonctions al{\'e}atoires
  gaussiennes}.
In \bbooktitle{Ecole d’Et{\'e} de Probabilit{\'e}s de Saint-Flour IV—1974}
\bpages{1--96}.
\bpublisher{Springer}.
\end{bincollection}
\endbibitem

\bibitem{fillard2007measuring}
\begin{barticle}[author]
\bauthor{\bsnm{Fillard},~\bfnm{Pierre}\binits{P.}},
  \bauthor{\bsnm{Arsigny},~\bfnm{Vincent}\binits{V.}},
  \bauthor{\bsnm{Pennec},~\bfnm{Xavier}\binits{X.}},
  \bauthor{\bsnm{Hayashi},~\bfnm{Kiralee~M.}\binits{K.~M.}},
  \bauthor{\bsnm{Thompson},~\bfnm{Paul~M.}\binits{P.~M.}} \AND
  \bauthor{\bsnm{Ayache},~\bfnm{Nicholas}\binits{N.}}
(\byear{2007}).
\btitle{Measuring brain variability by extrapolating sparse tensor fields
  measured on sulcal lines}.
\bjournal{Neuroimage}
\bvolume{34}
\bpages{639--650}.
\end{barticle}
\endbibitem

\bibitem{fillard2005extrapolation}
\begin{binproceedings}[author]
\bauthor{\bsnm{Fillard},~\bfnm{Pierre}\binits{P.}},
  \bauthor{\bsnm{Arsigny},~\bfnm{Vincent}\binits{V.}},
  \bauthor{\bsnm{Pennec},~\bfnm{Xavier}\binits{X.}},
  \bauthor{\bsnm{Thompson},~\bfnm{Paul~M.}\binits{P.~M.}} \AND
  \bauthor{\bsnm{Ayache},~\bfnm{Nicholas}\binits{N.}}
(\byear{2005}).
\btitle{Extrapolation of sparse tensor fields: Application to the modeling of
  brain variability}.
In \bbooktitle{Biennial International Conference on Information Processing in
  Medical Imaging}
\bpages{27--38}.
\bpublisher{Springer}.
\end{binproceedings}
\endbibitem

\bibitem{frechet1948elements}
\begin{barticle}[author]
\bauthor{\bsnm{Fr{\'e}chet},~\bfnm{Maurice}\binits{M.}}
(\byear{1948}).
\btitle{Les {\'e}l{\'e}ments al{\'e}atoires de nature quelconque dans un espace
  distanci{\'e}}.
\bjournal{Annales de l'institut Henri Poincar{\'e}}
\bvolume{10}
\bpages{215--310}.
\end{barticle}
\endbibitem

\bibitem{gallot1990riemannian}
\begin{bbook}[author]
\bauthor{\bsnm{Gallot},~\bfnm{Sylvestre}\binits{S.}},
  \bauthor{\bsnm{Hulin},~\bfnm{Dominique}\binits{D.}} \AND
  \bauthor{\bsnm{Lafontaine},~\bfnm{Jacques}\binits{J.}}
(\byear{1990}).
\btitle{Riemannian Geometry}
\bvolume{2}.
\bpublisher{Springer}.
\end{bbook}
\endbibitem

\bibitem{ganea2018hyperbolic}
\begin{binproceedings}[author]
\bauthor{\bsnm{Ganea},~\bfnm{Octavian-Eugen}\binits{O.-E.}},
  \bauthor{\bsnm{B{\'e}cigneul},~\bfnm{Gary}\binits{G.}} \AND
  \bauthor{\bsnm{Hofmann},~\bfnm{Thomas}\binits{T.}}
(\byear{2018}).
\btitle{Hyperbolic neural networks}.
In \bbooktitle{32nd Conference on Neural Information Processing Systems
  (NeurIPS 2018)}.
\end{binproceedings}
\endbibitem

\bibitem{ghodrati2021distribution}
\begin{barticle}[author]
\bauthor{\bsnm{Ghodrati},~\bfnm{Laya}\binits{L.}} \AND
  \bauthor{\bsnm{Panaretos},~\bfnm{Victor~M.}\binits{V.~M.}}
(\byear{2021}).
\btitle{Distribution-on-distribution regression via optimal transport maps}.
\bjournal{arXiv preprint arXiv:2104.09418}.
\end{barticle}
\endbibitem

\bibitem{gine2021mathematical}
\begin{bbook}[author]
\bauthor{\bsnm{Gin{\'e}},~\bfnm{Evarist}\binits{E.}} \AND
  \bauthor{\bsnm{Nickl},~\bfnm{Richard}\binits{R.}}
(\byear{2021}).
\btitle{Mathematical Foundations of Infinite-Dimensional Statistical Models}.
\bpublisher{Cambridge University Press}.
\end{bbook}
\endbibitem

\bibitem{hopkins2020mean}
\begin{barticle}[author]
\bauthor{\bsnm{Hopkins},~\bfnm{Samuel~B}\binits{S.~B.}}
(\byear{2020}).
\btitle{Mean estimation with sub-Gaussian rates in polynomial time}.
\bjournal{The Annals of Statistics}
\bvolume{48}
\bpages{1193--1213}.
\end{barticle}
\endbibitem

\bibitem{horvath2021monitoring}
\begin{barticle}[author]
\bauthor{\bsnm{Horv\'{a}th},~\bfnm{Lajos}\binits{L.}},
  \bauthor{\bsnm{Kokoszka},~\bfnm{Piotr}\binits{P.}} \AND
  \bauthor{\bsnm{Wang},~\bfnm{Shixuan}\binits{S.}}
(\byear{2021}).
\btitle{Monitoring for a change point in a sequence of distributions}.
\bjournal{The Annals of Statistics}
\bvolume{49}
\bpages{2271--2291}.
\end{barticle}
\endbibitem

\bibitem{hsu2016loss}
\begin{barticle}[author]
\bauthor{\bsnm{Hsu},~\bfnm{Daniel}\binits{D.}} \AND
  \bauthor{\bsnm{Sabato},~\bfnm{Sivan}\binits{S.}}
(\byear{2016}).
\btitle{Loss minimization and parameter estimation with heavy tails}.
\bjournal{The Journal of Machine Learning Research}
\bvolume{17}
\bpages{543--582}.
\end{barticle}
\endbibitem

\bibitem{kloeckner2010geometric}
\begin{barticle}[author]
\bauthor{\bsnm{Kloeckner},~\bfnm{Beno{\^i}t}\binits{B.}}
(\byear{2010}).
\btitle{A geometric study of Wasserstein spaces: Euclidean spaces}.
\bjournal{Annali della Scuola Normale Superiore di Pisa-Classe di Scienze}
\bvolume{9}
\bpages{297--323}.
\end{barticle}
\endbibitem

\bibitem{koltchinskii2011oracle}
\begin{bbook}[author]
\bauthor{\bsnm{Koltchinskii},~\bfnm{Vladimir}\binits{V.}}
(\byear{2011}).
\btitle{Oracle Inequalities in Empirical Risk Minimization and Sparse Recovery
  Problems: Ecole d’Et{\'e} de Probabilit{\'e}s de Saint-Flour XXXVIII-2008}
\bvolume{2033}.
\bpublisher{Springer Science \& Business Media}.
\end{bbook}
\endbibitem

\bibitem{le2012asymptotic}
\begin{bbook}[author]
\bauthor{\bsnm{Le~Cam},~\bfnm{Lucien}\binits{L.}}
(\byear{2012}).
\btitle{Asymptotic methods in statistical decision theory}.
\bpublisher{Springer Science \& Business Media}.
\end{bbook}
\endbibitem

\bibitem{le2017existence}
\begin{barticle}[author]
\bauthor{\bsnm{Le~Gouic},~\bfnm{Thibaut}\binits{T.}} \AND
  \bauthor{\bsnm{Loubes},~\bfnm{Jean-Michel}\binits{J.-M.}}
(\byear{2017}).
\btitle{Existence and consistency of Wasserstein barycenters}.
\bjournal{Probability Theory and Related Fields}
\bvolume{168}
\bpages{901--917}.
\end{barticle}
\endbibitem

\bibitem{lecue2020robust}
\begin{barticle}[author]
\bauthor{\bsnm{Lecu{\'e}},~\bfnm{Guillaume}\binits{G.}} \AND
  \bauthor{\bsnm{Lerasle},~\bfnm{Matthieu}\binits{M.}}
(\byear{2020}).
\btitle{Robust machine learning by median-of-means: theory and practice}.
\bjournal{The Annals of Statistics}
\bvolume{48}
\bpages{906--931}.
\end{barticle}
\endbibitem

\bibitem{lederer2014new}
\begin{barticle}[author]
\bauthor{\bsnm{Lederer},~\bfnm{Johannes}\binits{J.}} \AND \bauthor{\bsnm{Van
  De~Geer},~\bfnm{Sara}\binits{S.}}
(\byear{2014}).
\btitle{New concentration inequalities for suprema of empirical processes}.
\bjournal{Bernoulli}
\bvolume{20}
\bpages{2020--2038}.
\end{barticle}
\endbibitem

\bibitem{lei2020fast}
\begin{binproceedings}[author]
\bauthor{\bsnm{Lei},~\bfnm{Zhixian}\binits{Z.}},
  \bauthor{\bsnm{Luh},~\bfnm{Kyle}\binits{K.}},
  \bauthor{\bsnm{Venkat},~\bfnm{Prayaag}\binits{P.}} \AND
  \bauthor{\bsnm{Zhang},~\bfnm{Fred}\binits{F.}}
(\byear{2020}).
\btitle{A fast spectral algorithm for mean estimation with sub-gaussian rates}.
In \bbooktitle{Conference on Learning Theory}
\bpages{2598--2612}.
\bpublisher{PMLR}.
\end{binproceedings}
\endbibitem

\bibitem{lerasle2019monk}
\begin{binproceedings}[author]
\bauthor{\bsnm{Lerasle},~\bfnm{Matthieu}\binits{M.}},
  \bauthor{\bsnm{Szab{\'o}},~\bfnm{Zolt{\'a}n}\binits{Z.}},
  \bauthor{\bsnm{Mathieu},~\bfnm{Timoth{\'e}e}\binits{T.}} \AND
  \bauthor{\bsnm{Lecu{\'e}},~\bfnm{Guillaume}\binits{G.}}
(\byear{2019}).
\btitle{MONK outlier-robust mean embedding estimation by median-of-means}.
In \bbooktitle{International Conference on Machine Learning}
\bpages{3782--3793}.
\bpublisher{PMLR}.
\end{binproceedings}
\endbibitem

\bibitem{Lin2019}
\begin{barticle}[author]
\bauthor{\bsnm{Lin},~\bfnm{Zhenhua}\binits{Z.}}
(\byear{2019}).
\btitle{{R}iemannian geometry of symmetric positive definite matrices via
  {C}holesky decomposition}.
\bjournal{SIAM Journal of Matrix Analysis and Applications}
\bvolume{40}
\bpages{1353--1370}.
\end{barticle}
\endbibitem

\bibitem{lin2020additive}
\begin{barticle}[author]
\bauthor{\bsnm{Lin},~\bfnm{Zhenhua}\binits{Z.}},
  \bauthor{\bsnm{M{\"u}ller},~\bfnm{Hans-Georg}\binits{H.-G.}} \AND
  \bauthor{\bsnm{Park},~\bfnm{Byeong~U.}\binits{B.~U.}}
(\byear{2021}).
\btitle{Additive models for symmetric positive-definite matrices, Riemannian
  manifolds and Lie groups}.
\bjournal{arXiv preprint arXiv:2009.08789}.
\end{barticle}
\endbibitem

\bibitem{lugosi2019sub}
\begin{barticle}[author]
\bauthor{\bsnm{Lugosi},~\bfnm{G{\'a}bor}\binits{G.}} \AND
  \bauthor{\bsnm{Mendelson},~\bfnm{Shahar}\binits{S.}}
(\byear{2019}).
\btitle{Sub-Gaussian estimators of the mean of a random vector}.
\bjournal{The Annals of Statistics}
\bvolume{47}
\bpages{783--794}.
\end{barticle}
\endbibitem

\bibitem{lugosi2019mean}
\begin{barticle}[author]
\bauthor{\bsnm{Lugosi},~\bfnm{G{\'a}bor}\binits{G.}} \AND
  \bauthor{\bsnm{Mendelson},~\bfnm{Shahar}\binits{S.}}
(\byear{2019}).
\btitle{Mean estimation and regression under heavy-tailed distributions: A
  survey}.
\bjournal{Foundations of Computational Mathematics}
\bvolume{19}
\bpages{1145--1190}.
\end{barticle}
\endbibitem

\bibitem{lugosi2021robust}
\begin{barticle}[author]
\bauthor{\bsnm{Lugosi},~\bfnm{Gabor}\binits{G.}} \AND
  \bauthor{\bsnm{Mendelson},~\bfnm{Shahar}\binits{S.}}
(\byear{2021}).
\btitle{Robust multivariate mean estimation: the optimality of trimmed mean}.
\bjournal{The Annals of Statistics}
\bvolume{49}
\bpages{393--410}.
\end{barticle}
\endbibitem

\bibitem{minsker2015geometric}
\begin{barticle}[author]
\bauthor{\bsnm{Minsker},~\bfnm{Stanislav}\binits{S.}}
(\byear{2015}).
\btitle{Geometric median and robust estimation in Banach spaces}.
\bjournal{Bernoulli}
\bvolume{21}
\bpages{2308--2335}.
\end{barticle}
\endbibitem

\bibitem{nemirovskij1983problem}
\begin{barticle}[author]
\bauthor{\bsnm{Nemirovskij},~\bfnm{Arkadij~Semenovi{\v{c}}}\binits{A.~S.}} \AND
  \bauthor{\bsnm{Yudin},~\bfnm{David~Borisovich}\binits{D.~B.}}
(\byear{1983}).
\btitle{Problem Complexity and Method Efficiency in Optimization}.
\end{barticle}
\endbibitem

\bibitem{panaretos2020invitation}
\begin{bbook}[author]
\bauthor{\bsnm{Panaretos},~\bfnm{Victor~M.}\binits{V.~M.}} \AND
  \bauthor{\bsnm{Zemel},~\bfnm{Yoav}\binits{Y.}}
(\byear{2020}).
\btitle{An Invitation to Statistics in Wasserstein Space}.
\bpublisher{Springer Nature}.
\end{bbook}
\endbibitem

\bibitem{pennec2019riemannian}
\begin{bbook}[author]
\bauthor{\bsnm{Pennec},~\bfnm{Xavier}\binits{X.}},
  \bauthor{\bsnm{Sommer},~\bfnm{Stefan}\binits{S.}} \AND
  \bauthor{\bsnm{Fletcher},~\bfnm{Tom}\binits{T.}}
(\byear{2019}).
\btitle{Riemannian Geometric Statistics in Medical Image Analysis}.
\bpublisher{Academic Press}.
\end{bbook}
\endbibitem

\bibitem{petersen2019frechet}
\begin{barticle}[author]
\bauthor{\bsnm{Petersen},~\bfnm{Alexander}\binits{A.}} \AND
  \bauthor{\bsnm{M{\"u}ller},~\bfnm{Hans-Georg}\binits{H.-G.}}
(\byear{2019}).
\btitle{Fr{\'e}chet regression for random objects with Euclidean predictors}.
\bjournal{The Annals of Statistics}
\bvolume{47}
\bpages{691--719}.
\end{barticle}
\endbibitem

\bibitem{schotz2019convergence}
\begin{barticle}[author]
\bauthor{\bsnm{Sch{\"o}tz},~\bfnm{Christof}\binits{C.}}
(\byear{2019}).
\btitle{Convergence rates for the generalized Fr{\'e}chet mean via the
  quadruple inequality}.
\bjournal{Electronic Journal of Statistics}
\bvolume{13}
\bpages{4280--4345}.
\end{barticle}
\endbibitem

\bibitem{sturm2003probability}
\begin{barticle}[author]
\bauthor{\bsnm{Sturm},~\bfnm{Karl-Theodor}\binits{K.-T.}},
  \bauthor{\bsnm{Coulhon},~\bfnm{T}\binits{T.}} \AND
  \bauthor{\bsnm{Grigor’yan},~\bfnm{A}\binits{A.}}
(\byear{2003}).
\btitle{Probability measures on metric spaces of nonpositive curvature}.
\bjournal{Heat Kernels and Analysis on Manifolds, Graphs, and Metric Spaces,
  Contemporary Mathematics {\rm \bf 358}. {\rm American Mathematical Society}}.
\end{barticle}
\endbibitem

\bibitem{tifrea2018poincar}
\begin{barticle}[author]
\bauthor{\bsnm{Tifrea},~\bfnm{Alexandru}\binits{A.}},
  \bauthor{\bsnm{B{\'e}cigneul},~\bfnm{Gary}\binits{G.}} \AND
  \bauthor{\bsnm{Ganea},~\bfnm{Octavian-Eugen}\binits{O.-E.}}
(\byear{2018}).
\btitle{Poincar{\'e} glove: Hyperbolic word embeddings}.
\bjournal{arXiv preprint arXiv:1810.06546}.
\end{barticle}
\endbibitem

\bibitem{van2013bernstein}
\begin{barticle}[author]
\bauthor{\bparticle{van~de} \bsnm{Geer},~\bfnm{Sara}\binits{S.}} \AND
  \bauthor{\bsnm{Lederer},~\bfnm{Johannes}\binits{J.}}
(\byear{2013}).
\btitle{The Bernstein--Orlicz norm and deviation inequalities}.
\bjournal{Probability Theory and Related Fields}
\bvolume{157}
\bpages{225--250}.
\end{barticle}
\endbibitem

\bibitem{van2014probability}
\begin{btechreport}[author]
\bauthor{\bsnm{Van~Handel},~\bfnm{Ramon}\binits{R.}}
(\byear{2014}).
\btitle{Probability in high dimension}
\btype{Technical Report},
\bpublisher{Princeton Univ NJ}.
\end{btechreport}
\endbibitem

\bibitem{villani2009optimal}
\begin{bbook}[author]
\bauthor{\bsnm{Villani},~\bfnm{C{\'e}dric}\binits{C.}}
(\byear{2009}).
\btitle{Optimal Transport: Old and New, Grundlehren der mathematischen
  Wissenschaften {\rm \bf 338}}.
\bpublisher{Springer}.
\end{bbook}
\endbibitem

\bibitem{zhang2021wasserstein}
\begin{barticle}[author]
\bauthor{\bsnm{Zhang},~\bfnm{Chao}\binits{C.}},
  \bauthor{\bsnm{Kokoszka},~\bfnm{Piotr}\binits{P.}} \AND
  \bauthor{\bsnm{Petersen},~\bfnm{Alexander}\binits{A.}}
(\byear{2022}).
\btitle{Wasserstein autoregressive models for density time series}.
\bjournal{Journal of Time Series Analysis}
\bvolume{43}
\bpages{30--52}.
\end{barticle}
\endbibitem

\end{thebibliography}
\newpage

\section*{}

\renewcommand{\theequation}{S.\arabic{equation}}
\renewcommand{\thesubsection}{S.\arabic{subsection}}
\renewcommand{\thelemma}{S.\arabic{lemma}}
\renewcommand{\theproposition}{S.\arabic{proposition}}
\setcounter{page}{1}

\setcounter{equation}{0}
\setcounter{subsection}{0}
\setcounter{proposition}{0}

\centerline{\bf \large Supplementary Material to}
\centerline{\bf \large `Exponential Concentration for Geometric-Median-of-means}
\centerline{\bf \large in Non-Positive Curvature Spaces'}
\centerline{\bf \large by H. Yun and B. U. Park}

\bigskip

In this Supplementary Material, we give an additional proposition and prove the propositions in the main paper.

\subsection{Additional proposition}

\begin{proposition}\label{prop_power_mean_polygon}
Let $P$ be a probability measure in $\mathbb{R}^{D}$ and $\eta(x,y)=\psi(|x-y|)$ where $\psi:[0,+\infty)\rightarrow \mathbb{R}$ is strictly increasing and convex. Assume that there exists $z \in \mathbb{R}^{D}$ such that
\begin{equation*}
P(A)=P\left(Q(A-z)+z\right),\quad A \in \mathscr{B}(\mathbb{R}^{D})
\end{equation*}
for an orthogonal matrix $Q \in \mathbb{R}^{D \times D}$ with $I+Q+\dots+Q^{m-1}=0$ for some
integer $m\ge 2$.
Then, $z$ is the unique Fr\'echet mean with respect to $\eta:\mathbb{R}^{D} \times \mathbb{R}^{D} \rightarrow \mathbb{R}$.
\end{proposition}

\begin{proof}
We write $d(x,y)=|x-y|$. Since $Q \in \mathbb{R}^{D \times D}$ is an orthogonal matrix, it holds
that, for any $x, y \in \mathbb{R}^{D}$,
\begin{equation}\label{orthoequal}
\sum^{m-1}_{j=0} d(x, Q^j y)=\sum^{m-1}_{j=0} d(Qx, Q^j y)=\cdots=\sum^{m-1}_{j=0} d(Q^{m-1} x, Q^{j} y).
\end{equation}
By \eqref{orthoequal} and the subadditivity of the Euclidean norm, we get
\begin{align*}
\sum^{m-1}_{j=0} d(x, Q^{j} y)&=\frac{1}{m} \sum^{m-1}_{l=0} \sum^{m-1}_{j=0} d(Q^{l}x, Q^j y)
\geq \sum^{m-1}_{j=0} d \left(\frac{1}{m} \sum^{m-1}_{l=0} Q^{l}x, Q^{j} y \right)=\sum^{m-1}_{j=0} d(0, Q^{j} y).
\end{align*}
Now, by Jensen's inequality,
\begin{align*}
    \frac{1}{m} \sum^{m-1}_{j=0} \psi(d(x, Q^{j} y))
    \geq \psi \left(\frac{1}{m} \sum^{m-1}_{j=0} d(x, Q^{j} y) \right)
    \geq \psi \left(\frac{1}{m} \sum^{m-1}_{j=0} d(0, Q^{j} y) \right)
\end{align*}
and the equality holds if and only if $x=0$.
Considering translation, for any $x, y \in \mathbb{R}^{D}$,
\begin{equation*}
\sum^{m-1}_{j=0} \psi \left(d(x, Q^{j} (y-z)+z) \right) \geq \sum^{m-1}_{j=0} \psi \left(d(z, Q^{j} (y-z)+z) \right) \end{equation*}
and the equality holds if and only if $x=z$.
Therefore,
\begin{align*}
\int_{\mathbb{R}^{D}} \eta(x, y) \mathrm{d} P(y)
&= \frac{1}{m} \int_{\mathbb{R}^{D}} \sum^{m-1}_{j=0} \psi \left(d(x, Q^{j} (y-z)+z) \right) \mathrm{d} P(y) \\
&\geq \frac{1}{m} \int_{\mathbb{R}^{D}} \sum^{m-1}_{j=0} \psi \left(d(z, Q^j (y-z)+z) \right) \mathrm{d} P(y)
=\int_{\mathbb{R}^{D}} \eta(d(z, y)) \mathrm{d} P(y)
\end{align*}
and the equality holds if and only if $x=z$.
\end{proof}

\subsection{Some lemmas}

To provide an upper bound to the right hand side of \eqref{empi_var_ineq} with high probability, we need a tail inequality for empirical processes.
In our setup, $\|\eta(x, \cdot)-\eta(x^{*}, \cdot)\|_\infty$ may be unbounded as $x$ moves. Under some strong condition on the tail of $P$, one may be able to
obtain an exponential tail inequality, see \cite{adamczak2008tail, van2013bernstein}. Since we assume only a finite second moment of $P$,
we use the following polynomial tail inequality.

\begin{lemma}[\cite{lederer2014new}]\label{Lemma1}
Let $X_{1}, \ldots, X_{n}$ be i.i.d. copies of $X$ taking values in a measurable space $(\mathcal{S}, \mathcal{B})$ with probability measure $P$,
and let $\mathcal{G}$ be a countable class of measurable functions $f: \mathcal{S} \rightarrow \mathbb{R}$ with $Pf=0$.
Put $Z=\sup_{f \in \mathcal{G}}(P-P_{n})f$ and $\sigma^{2}=\sup_{f \in \mathcal{G}}Pf^{2}$.
Assume that the envelope $H$ of the class $\mathcal{G}$ satisfies $\mathbb{E}(H^{p}) \leq M^p$
for some $p\ge 1$ and $M>0$. Then, for any $\varepsilon>0$, it holds that
\begin{equation*}
\mathbb{P}\,\left(Z \geq 4\,\mathbb{E}(Z)+\varepsilon\right) \leq \min _{1 \leq l \leq p} \frac{l \cdot \Gamma(l/2) \left(\sqrt{32/n} M\right)^{l}}{\varepsilon^{l}}.
\end{equation*}
If $\mathbb{E}(H^2) \leq M^2$, in particular, we get that, for any $\Delta \in (0,1)$,
\begin{equation*}
\mathbb{P}\,\left(Z \leq 4\,\mathbb{E}(Z)+\frac{8M}{\sqrt{n \Delta}}\right) \geq 1-\Delta.
\end{equation*}
\end{lemma}

Below, we present two more lemmas for the proof of the theorems.
Recall the definition of $H_\delta \equiv H_{\delta,\eta}$ given at
\eqref{env-bound}, which envelops $\mathcal{F}(\delta)\equiv\mathcal{F}_\eta(\delta)$.

\begin{lemma}\label{Lemma2}
Let $\eta:\mathcal{M}\times\mathcal{M}\rightarrow\mathbb{R}$ be a measurable function and $X$ an $\mathcal{M}$-valued random element
with Fr\'echet mean $x^*$ and covariance $\sigma_{X}^{2}$. Let $\delta>0$.
Then, under the assumptions (A1) and (A2),
\begin{equation*}
\sigma(\delta) \leq \bar\sigma(\delta), \quad \mathbb{E}\big(H_\delta(X_{1})^2\big) \leq \bar\sigma(\delta)^{2}, \quad \mathbb{E} \left(\|H_\delta\|^{2}_{2,P_{n}}\right) \leq \bar\sigma(\delta)^{2},
\end{equation*}
where $\bar\sigma(\delta)=4\sqrt{K\sigma_{X}^{2}\delta^{\beta}}$.
\end{lemma}

\begin{proof}
Recall the definition of $H_\delta$ at \eqref{env-bound}. Then,
\begin{align*}
    \max \left\{\sigma^{2}(\delta), \mathbb{E}H_{\delta}(X_{1})^2 \right\}
    &\leq 4K \delta^{\beta} \int_{\mathcal{M}}\left(\int_{\mathcal{M}}d(y,z) \,\rd P(z)\right)^{2} \,\rd P(y) \\
    &\leq 4K \delta^{\beta} \int_{\mathcal{M}}\int_{\mathcal{M}}d(y,z)^{2} \,\rd P(z) \,\rd P(y) \\
    &\leq 8 K\delta^{\beta} \int_{\mathcal{M}}\int_{\mathcal{M}}\left(d(y,x^*)^{2}+d(z,x^*)^{2}\right) \,\rd P(z) \,\rd P(y) \\
    &=16K\sigma_{X}^{2}\delta^{\beta}.
\end{align*}
Now, let $X_{1}',\dots,X_{n}'$ be an independent copy of $X_{1},\dots,X_{n}$. By the triangular inequality, it holds that
\begin{align*}
    \mathbb{E} \left(\|H_{\delta}\|^{2}_{2,P_{n}}\right)
    &=\mathbb{E} \left( \frac{1}{n} \sum^{n}_{i=1}H_{\delta}(X_{i})^{2}\right)\\
    &=\frac{1}{n} \sum^{n}_{i=1} 4 K\delta^{\beta} \mathbb{E} \left( \int_{\mathcal{M}}d(X_{i},z) \,\rd P(z) \right)^{2} \\
    &\le \frac{4 K\delta^{\beta}}{n} \sum^{n}_{i=1}  \mathbb{E} \left( \int_{\mathcal{M}}d(X_{i},z)^{2} \,\rd P(z) \right)\\
    &=\frac{4 K\delta^{\beta}}{n} \sum^{n}_{i=1}  \mathbb{E} \left( d(X_{i},X'_{i})^{2} \right) \\
    &=4 K\delta^{\beta}  \,\mathbb{E} \left( d(X_{1},X'_{1})^{2} \right)\\
    &\leq 8 K\delta^{\beta} \,\mathbb{E}\left(d(X_{1},x^*)^{2}+d(X'_{1},x^*)^{2}\right) \\
    &=16K\sigma_{X}^{2}\delta^{\beta}.
\end{align*}
\end{proof}

The following lemma provides an improved chaining bound for Gaussian processes. For a proof, see Theorem 5.31 in \cite{van2014probability} or Lemma 5.1 in \cite{ahidar2019rate}.

\begin{lemma}\label{Lemma3}
Let $(X_{t})_{t \in \mathcal{F}}$ be a real-valued process indexed by a pseudo metric space $(\mathcal{F}, d)$
with the following properties: (i) there exists a countable subset $\mathcal{F}^{\prime} \subset \mathcal{F}$ such that
$X_{t}=\lim _{s \rightarrow t, s \in \mathcal{F}^{\prime}} X_{s}\text { a.s.}$ for any $t \in \mathcal{F}$;
(ii) $X_{t}$ is sub-Gaussian, i.e.
\[
\log \mathbb{E}\big(e^{\theta\left(X_{s}-X_{t}\right)}\big) \leq \theta^{2} d(s, t)^{2}/2
\]
for any $s,t \in \mathcal{F}$ and $\theta \in \mathbb{R}$;
(iii) there exists a random variable $L$ such that $|X_{s}-X_{t}| \leq L \,d(s, t)\text { a.s.}$
for all $s,t \in \mathcal{F}$.
Then, for any $S \subset \mathcal{F}$ and any $\varepsilon \geq 0$, it holds that
\begin{equation*}
\mathbb{E} \big(\sup _{t \in S} X_{t}\big) \leq 2 \,\varepsilon \,\mathbb{E}(L)+12 \int_{\varepsilon}^{+\infty} \sqrt{\log N(u,\mathcal{F}, d)}\, \mathrm{d} u.
\end{equation*}
\end{lemma}

\subsection{Proofs of propositions in \cref{Section 4}}

\begin{proof}[Proof of \cref{thm_gen_CN}]
Since $1/2 \leq \alpha/2 \leq 1$, we have $a^{\alpha/2}+b^{\alpha/2} \geq (a+b)^{\alpha/2}$ for any $a,b \geq 0$, so that
\begin{equation}\label{jensen_geodsc}
\begin{split}
&(1+\delta)^{1-\alpha/2} \left\{(1-t)^{\alpha/2} d(z,\gamma_{0})^{\alpha}+t^{\alpha/2} d(z,\gamma_{1})^{\alpha} \right\}- d(z,\gamma_{t})^{\alpha}\\
&\qquad \geq (1+\delta)^{1-\alpha/2} \left\{(1-t) d(z,\gamma_{0})^{2}+t d(z,\gamma_{1})^{2} \right\}^{\alpha/2}- d(z,\gamma_{t})^{\alpha}.
\end{split}
\end{equation}
An application of H\"older's inequality gives
\[
\left( \delta_{1}+\delta_{2}\right)^{1-\alpha/2} \left( a_{1}+a_{2}\right)^{\alpha/2}
\;\geq\; \delta_{1}^{1-\alpha/2} a_{1}^{\alpha/2}+\delta_{2}^{1-\alpha/2} a_{2}^{\alpha/2}
\]
for all $\delta_i,a_i \geq 0$ and $\alpha \in (0,2)$. The above inequality also holds for $\alpha=0$ and $2$.
Applying the inequality with $\delta_{1}=1, \delta_{2}=\delta, a_{1}=d(z,\gamma_{t})^{2}, a_{1}+a_{2}=(1-t) d(z,\gamma_{0})^{2}+t d(z,\gamma_{1})^{2}$
to the right hand side of the inequality at \eqref{jensen_geodsc}, we get
\begin{equation*}
\begin{split}
&(1+\delta)^{1-\alpha/2} \left((1-t)^{\alpha/2} d(z,\gamma_{0})^{\alpha}+t^{\alpha/2} d(z,\gamma_{1})^{\alpha} \right)- d(z,\gamma_{t})^{\alpha}\\
&\qquad \geq \delta^{1-\alpha/2} \left((1-t) d(z,\gamma_{0})^{2}+t d(z,\gamma_{1})^{2}- d(z,\gamma_{t})^{2}\right)^{\alpha/2}\\
&\qquad \geq \delta^{1-\alpha/2} \left(t(1-t)d(\gamma_{0}, \gamma_{1})^{2}\right)^{\alpha/2}.
\end{split}
\end{equation*}
We note that $a_2\ge 0$ and the last inequality follows from the CN inequality.
\end{proof}

\begin{proof}[Proof of \cref{prop_gen_var_ineq}]
For $x \in \cM\setminus\{x^*\}$, we apply \cref{thm_gen_CN} to $\delta=0$ and $\gamma$ being the geodesic $\gamma^{x}:[0,1] \to \mathcal{M}$
with $\gamma_0^x=x^*$ and $\gamma_1^x=x$, and
then integrate both sides of the inequality with respect to $z$. This gives
\begin{equation}\label{gencn_ineq}
(1-t)^{\alpha/2} F_{\alpha}(x^*)+t^{\alpha/2} F_{\alpha}(x)- F_{\alpha}(\gamma^{x}_{t}) \geq 0
\end{equation}
for any $0 \leq t \leq 1$. Take an arbitrary $\varepsilon>0$.
By the definition of $b_{\alpha}(x)$, it holds that, for any $x \in \mathcal{M}\setminus\{x^*\}$, there exists $t\equiv t(x)>0$ such that
\begin{equation}\label{sup_assump}
F_{\alpha}(\gamma^{x}_{t})-\big(t^{\alpha/2}+(1-t)^{\alpha/2}\big) F_{\alpha}(x^*) \geq (b_{\alpha}(x)-\varepsilon) \,t^{\frac{\alpha}{2}} \,d(x,x^*)^{\alpha}.
\end{equation}
From \eqref{gencn_ineq} and \eqref{sup_assump}, it follows that
\begin{align*}
t^{\alpha/2}\left(F_{\alpha}(x)-F_{\alpha}(x^*) \right) \geq (b_{\alpha}(x)-\varepsilon)  t^{\alpha/2} d(x,x^*)^{\alpha},
\end{align*}
so that $F_{\alpha}(x)-F_{\alpha}(x^*) \geq (b_{\alpha}(x)-\varepsilon) d(x,x^*)^{\alpha}$. Since $\varepsilon>0$ was arbitrarily chosen, we have
\begin{equation*}
   F_{\alpha}(x)-F_{\alpha}(x^*) \geq b_{\alpha}(x) \cdot d(x,x^*)^{\alpha},
\end{equation*}
which completes the proof of the proposition.
\end{proof}

\begin{proof}[Proof of \cref{prop_npc_cover}]
Recall $H_\delta(y)=2\sqrt{K \delta^{\beta}} \int_{\mathcal{M}}d(y,z) \,\rd P(z)$ from \eqref{env-bound}.
Now, for $x,y \in \cM(\delta)$,
\begin{align*}
&\|f(x,\cdot)-f(y,\cdot)\|_{2,P_{n}}^{2}\\
&\qquad =\frac{1}{n} \sum_{i=1}^{n} \left(\eta_c(x,X_{i})-\eta_c(y,X_{i})\right)^{2} \\
&\qquad =\frac{1}{n} \sum_{i=1}^{n} \left(\int_{\mathcal{M}}\left(d(x,X_{i})^{\alpha}-d(y,X_{i})^{\alpha}-d(x,z)^{\alpha}+d(y,z)^{\alpha}\right) \rd P(z) \right)^{2}\\
&\qquad \leq \frac{\alpha^{2} 2^{-2\alpha+4}}{n} \cdot d(x,y)^{2\alpha-2} \sum_{i=1}^{n} \left(\int_{\mathcal{M}} d(X_{i},z)\rd P(z) \right)^{2}\\
&\qquad =\frac{\alpha^{2} 2^{-2\alpha+2}}{K \delta^{\beta}} \cdot d(x,y)^{2\alpha-2} \|H_\delta\|^{2}_{2,P_{n}},
\end{align*}
where the inequality follows from \eqref{power-quad-ineq}.
Since $l(\cdot,\cdot)=\alpha2^{-\alpha+1}d(\cdot,\cdot)^{\alpha-1}$ and $\beta=2-2/\alpha$,
\begin{equation*}
    \cM(\delta) \subset B \left( x^*, \left( \frac{K \delta^{\beta}}{\alpha^{2} 2^{-2\alpha+2}}\right)^{\frac{1}{2(\alpha-1)}}\right).
\end{equation*}
Thus, it holds that
\begin{align*}
&N\left(\tau\|H_\delta\|_{2,P_{n}}, \mathcal{F}(\delta), \|\cdot\|_{2,P_{n}} \right)\\
&\qquad \leq N\left(\left( \frac{\tau^{2} K \delta^{\beta}}{\alpha^{2} 2^{-2\alpha+2}}\right)^{\frac{1}{2(\alpha-1)}}, \,\cM(\delta), d  \right) \\
&\qquad \leq N\left(\left( \frac{\tau^{2} K \delta^{\beta}}{\alpha^{2} 2^{-2\alpha+2}}\right)^{\frac{1}{2(\alpha-1)}}, \,B \left( x^*, \left( \frac{K \delta^{\beta}}{\alpha^{2} 2^{-2\alpha+2}}\right)^{\frac{1}{2(\alpha-1)}}\right), d \right)\\
&\qquad \leq \left(\frac{A_{1}}{\tau^{1/(\alpha-1)}}\right)^{D_{1}}.
\end{align*}
\end{proof}

\begin{proof}[Proof of \cref{prop_hilb_cover}]
Recall from \cref{hilb_example2} that $\cM(\delta)=B(x^*,\sqrt{\delta}),
\,\|f(x,\cdot)-f(y,\cdot)\|^{2}_{2,P}=4 \,\Sigma_{X} (x-y, x-y)$.
Also, $\mathcal{F}(\delta)=\big\{2\langle x-x^{*},x^{*}-\cdot \rangle: x \in B(x^*,\sqrt{\delta})\big\}$ and
$\sup_{x\in\mathcal{F}(\delta)}f(x,\cdot)=2\sqrt{\delta} \|\cdot-x^*\|$ is the envelope of the class $\mathcal{F}(\delta)$.
By Sudakov's minorisation (see Theorem 2.4.12. in \cite{gine2021mathematical} and also \cite{fernique1975regularite} for the specified constant),
\begin{equation*}
\log N\left(\tau, \mathcal{F}(\delta), \|\cdot\|_{2,P}  \right) \leq \frac{1}{8} \left(\frac{\mathbb{E}_{g} \big(\Sigma_{X}^{1/2}(g, g)\big)}{\tau/\sqrt{\delta}}\right)^{2}
\leq \frac{\operatorname{tr}(\Sigma_{X}) \delta}{8 \tau ^{2}}
\end{equation*}
where $g$ is a standard Gaussian random element taking values in $(\mathcal{H},d)$.
Since $\|H_\delta\|^{2}_{2,P}=4 \,\delta \,\mathbb{E} \,(\langle X-x^{*},X-x^{*} \rangle)=4 \,\delta \operatorname{tr}(\Sigma_{X})$, we have
\begin{equation*}
\log N\left(\tau\|H_{\delta}\|_{2,P}, \mathcal{F}(\delta), \| \cdot \|_{2,P} \right) \leq \frac{1}{32 \tau^{2}}
\end{equation*}
With the same machinery, one can also deduce the same result for the empirical measure $P_{n}$.
\end{proof}

\end{document}